\documentclass[sigconf]{acmart} 

\copyrightyear{2023} 
\acmYear{2023} 
\setcopyright{acmlicensed}
\acmConference[ISSAC 2023]{International Symposium on Symbolic and Algebraic Computation 2023}{July 24--27, 2023}{Tromsø, Norway}
\acmBooktitle{International Symposium on Symbolic and Algebraic Computation 2023 (ISSAC 2023), July 24--27, 2023, Tromsø, Norway}
\acmPrice{15.00}
\acmDOI{10.1145/3597066.3597071}
\acmISBN{979-8-4007-0039-2/23/07}

\usepackage[utf8]{inputenc}
\usepackage[T1]{fontenc}
\usepackage{amsmath,amsthm}
\usepackage{xspace}

\DeclareFontFamily{U}{mathb}{\hyphenchar\font45}
\DeclareFontShape{U}{mathb}{m}{n}{ <-6> mathb5 <6-7> mathb6 <7-8>
  mathb7 <8-9> mathb8 <9-10> mathb9 <10-12> mathb10 <12-> mathb12 }{}
\DeclareSymbolFont{mathb}{U}{mathb}{m}{n}
\DeclareMathSymbol{\prec}{\mathrel}{mathb}{"A0}
\DeclareMathSymbol{\succ}{\mathrel}{mathb}{"A1}
\DeclareMathSymbol{\preceq}{\mathrel}{mathb}{"A8}
\DeclareMathSymbol{\succeq}{\mathrel}{mathb}{"A9}
\DeclareMathSymbol{\precneq}{\mathrel}{mathb}{"AC}
\DeclareMathSymbol{\succneq}{\mathrel}{mathb}{"AD}

\usepackage{color}
\definecolor{gray}{gray}{0.4}

\usepackage{amscd}
\usepackage{natbib}
\usepackage{enumerate}

\usepackage[font=small]{caption}
\usepackage{subcaption}

\usepackage[english]{babel}
\usepackage[shortlabels,inline]{enumitem}
\setlist{leftmargin=*}

\usepackage[ruled,vlined,english,linesnumbered]{algorithm2e}
\SetArgSty{textup}
\SetKwInOut{Input}{Input}\SetKwInOut{Output}{Output}
\SetKwIF{If}{ElseIf}{Else}{if}{:}{else if}{else:}{end}%
\SetKw{KwTo}{in}
\SetKwFor{For}{for}{\string:}{end}%
\SetKwFor{While}{while}{:}{end}%
\SetKwFor{Forall}{forall}{:}{end}%
\AlgoDontDisplayBlockMarkers
\SetAlgoNoEnd

\usepackage{tikz}

\usepackage{fixme}
\fxsetup{author={T},envlayout={color},%
  layout={footnote,marginclue},theme=color}

\FXRegisterAuthor{ch}{ach}{C}

\usepackage{lipsum}

\usepackage[mathscr]{euscript}

\usepackage{array}

\usepackage{mathtools}
\mathtoolsset{showonlyrefs,showmanualtags}

\usepackage{tabu}
\usepackage{multirow}
\usepackage{booktabs}

\usepackage{siunitx}

\newtheorem{theorem}{Theorem}[section]
\newtheorem{lemma}[theorem]{Lemma}
\newtheorem{proposition}[theorem]{Proposition}

\newtheorem{corollary}[theorem]{Corollary}
\newtheorem{definition}[theorem]{Definition}

\newtheorem{remark}[theorem]{Remark}

\newcommand{\LM}{\mathrm{lm}}
\newcommand{\LT}{\mathrm{lt}}

\newcommand{\ZZ}{\mathbb{Z}}
\newcommand{\QQ}{\mathbb{Q}}

\def\Sig{^{[\Sigma]}}

\newcommand{\sig}{\mathrm{sig}}
\newcommand{\s}{\sig}

\newcommand{\ind}{\mathrm{ind}}

\newcommand{\divides}{\mid}
\newcommand{\LC}{\mathrm{lc}}

\newcommand{\lm}{\LM}
\newcommand{\lt}{\LT}
\newcommand{\lc}{\LC}
\newcommand{\amb}{\textup{amb}}
\newcommand{\lma}{\textup{\textsc{lm}}}
\newcommand{\siga}{\textup{\textsc{sig}}}

\newcommand{\lcmlc}{\textup{lcmlc}}

\newcommand{\lcm}{\textup{lcm}}
\newcommand{\sigComb}{\textup{sig-Comb}}
\newcommand{\Syz}{\textnormal{Syz}}

\newcommand{\SPol}{\textup{S-Pol}}
\newcommand{\GPol}{\textup{G-Pol}}

\def\<#1>{\langle#1\rangle}

\newcommand{\eg}{e.g.}
\newcommand{\ie}{i.e.\xspace}

\makeatletter
\newcommand{\removelatexerror}{\let\@latex@error\@gobble}
\makeatother

\newcommand{\RX}[1][X]{R\langle #1\rangle}

\newcommand{\RXY}{R[X]\langle Y\rangle}

\begin{document}

\title{Signature Gröbner bases in free algebras over rings}

\author{Clemens Hofstadler}
\authornote{C. Hofstadler was supported by the Austrian Science Fund (FWF) grant P~32301.}
\email{clemens.hofstadler@mathematik.uni-kassel.de}
\orcid{0000-0002-3025-0604}
\affiliation{%
  \institution{Institute of Mathematics, University of Kassel}
  \streetaddress{Heinrich-Plett-Stra\ss e 40 }
  \city{Kassel}
  \country{Germany}
}

\author{Thibaut Verron}
\authornote{T. Verron was supported by the Austrian Science Fund (FWF) grant P~34872.}
\email{thibaut.verron@jku.at}
\orcid{0000-0003-0087-9097}
\affiliation{%
  \institution{Institute for Algebra, Johannes Kepler University}
  \streetaddress{Altenberger Stra\ss e 69}
  \city{Linz}
  \country{Austria}
}

\begin{abstract}
  We generalize signature Gröbner bases, previously studied in the free algebra over a field or polynomial rings over a ring, to ideals in the \emph{mixed algebra} $R[x_{1},\dots,x_{k}]\langle y_{1},\dots,y_n \rangle$ where $R$ is a principal ideal domain.
  We give an algorithm for computing them, combining elements from the theory of commutative and noncommutative (signature) Gröbner bases, and prove its correctness.
  
  Applications include extensions of the free algebra with commutative variables, \eg, for homogenization purposes or for performing ideal theoretic operations such as intersections, and computations over $\ZZ$ as universal proofs over fields of arbitrary characteristic.
  
  By extending the signature cover criterion to our setting,
  our algorithm also lifts some technical restrictions from previous noncommutative signature-based algorithms, now allowing \eg, elimination orderings.
  We provide a prototype implementation for the case when $R$ is a field, and show that our algorithm for the mixed algebra is more efficient than classical approaches using existing algorithms.
\end{abstract}



\maketitle


\section{Introduction}
\label{sec:introduction}

Gröbner bases are an essential tool in computational algebra.
They are best known in algebras of commutative polynomials over fields~\cite{buchberger-1965}, but
have also been extended to various settings, including polynomials over rings~\cite{Moller:1988:grobnerrings2,Kandri-Rody-Kapur,Pan:Dbases,Lichtblau}, modules~\cite{Moller:1986:generalizationtomodules}, skew polynomial rings~\cite{Galligo-1985-some-algorithmic-questions}, and noncommutative polynomials in the free algebra over fields~\cite{Mor85} or over rings~\cite{Pritchard, mikhalev, Spic, NG4, Cameroun, labOrE, LMAZ20, SPES4}.
In this last setting, Gröbner bases have proven to be a key tool in reducing and proving operator identities~\cite{HW94,HRR19,CHRR20,SL20,raab2021formal}.

In parallel to those generalizations, the last 20 years have seen the development of a new paradigm for computing Gröbner bases, with the concept of signatures~\cite{Faugere-2002-F5,eder:2017:survey,Lairez-2022-AxiomsForTheory}.
This approach, which improves upon the idea of tracing syzygies presented in the earlier work~\cite{MMT}, has had a significant impact on the field, leading to advances also in other algorithmic tools such as staggered linear bases~\cite{HJ,Hashemi}, initially introduced in~\cite{GM2}. 

Beyond their original purpose of optimizing the algorithms, it was more recently observed that the data of signatures also has direct applications.
For instance, they have been used in computational geometry~\cite{eder2022signature}, and they allow to perform a number of operations on the syzygy module of a family of polynomials, without the additional cost of module Gröbner bases computations~\cite{GVW}.
In particular, the data of a Gröbner basis with its signatures makes it possible to easily get a \emph{proof} that an element lies in an ideal, by reconstructing a representation in terms of the original generators~\cite{hofstadler-verron-shortrep}.
This potential has led to signature Gröbner bases being generalized beyond commutative polynomials over fields, for instance to polynomials over rings~\cite{eder-pfister-popescu,francis-verron}, to solvable noncommutative algebras~\cite{Sun-2012-signature-groebner-solvable}, and recently to the free algebra~\cite{hofstadler-verron}.

In this paper, we consider two generalizations of signature Gröbner bases in the free algebra at once.
First, we consider the case of noncommutative polynomials with some commutative variables, namely elements in $R[X]\<Y> = R[x_1,\dots,x_{k}]\langle y_1,\dots,y_n\rangle$.
Such objects arise, for example, when introducing auxiliary variables such as homogenization variables or parameters (\eg, when computing the intersection of ideals or the homogeneous part of an ideal).
The naive approach, consisting of adding the commutator relations to the input polynomials, is inefficient with signatures, because the information of these relations is not propagated to the signatures.

In addition, we relax the conditions on the base ring, no longer requiring it to be a field, but only a principal ideal domain (PID).
This, in particular, allows computations over $\ZZ$, which can be considered universal as they remain valid in arbitrary characteristic.
Furthermore, $R[X]\<Y>$ provides a natural setting for studying many finitely presented structures, such as Iwahori-Hecke algebras~\cite{hump90,LMAZ20} or (discrete) Heisenberg groups~\cite{lind2015survey}.

In the noncommutative case, a difficulty is that most ideals do not admit a finite Gröbner basis, and even among those that do, most do not admit a finite Gröbner basis of the module of syzygies.
In~\cite{hofstadler-verron}, the authors proposed an algorithm taking advantage of the data encoded within signatures, and of the structure of the syzygy module, to define and compute signature Gröbner bases which may be finite.
This is the combined effect of two properties of noncommutative polynomials: the definition of S-polynomials~\cite{ber78} implicitly includes Buchberger's coprime criterion, which ensures that one can only form finitely many S-polynomials with any given pair of polynomials; and the signature F5 criterion makes it possible to identify a large number of structural syzygies.

In the case without signatures, when attempting to generalize Gröbner bases, it is frequent that, for structural reasons, Buchberger's criterion can no longer eliminate all but finitely many S\nobreakdash-polynomials.
This was already observed for classical noncommutative Gröbner bases over rings~\cite{mikhalev, LMAZ20}, and also carries over to the generalizations at hand: 
Buchberger's criterion requires also the coefficients to be coprime, and the subsequently weakened criterion cannot ensure that only finitely many
S-polynomials remain.
In addition, when working with Gröbner bases over PIDs, one needs to compute so-called G-polynomials for each pair of polynomials, and there will typically be infinitely many such G-polynomials.

In this paper, we describe an algorithm for computing signature Gröbner bases for polynomials in the mixed algebra $R[X]\langle Y\rangle$.
The algorithms and data structures combine elements from~\cite{mikhalev} for the combinations and reductions, \cite{hofstadler-verron} for managing signatures in the free algebra, and \cite{francis-verron} for managing signatures with coefficients in a PID.
With the limitations exposed above, one cannot hope that the algorithm terminates, but we show that it correctly enumerates a signature Gröbner basis, as well as a basis of the module of syzygies.
The proof combines elements from the theory of signatures over the free algebra and of signatures over PIDs in the commutative case.

A feature of the algorithm is that it relies on the signature cover criterion~\cite{GVW,francis-verron} 
instead of requiring that every element reduces to 0.
As in the commutative case, this allows to decouple the selection strategy used for selecting S- and G-polynomials from the order on the signatures.
As a consequence, the algorithm does not have to process S- and G-polynomials by increasing signature.
This is particularly important in the noncommutative case, where all algorithms require that the selection strategy is so-called fair.
This is a strong requirement excluding for instance elimination orderings. 
As such, the new algorithm with the selection strategy decoupled from the signature order is the first algorithm that allows to enumerate a signature Gröbner basis in the free algebra for any monomial order.

\section{The mixed algebra}
\label{sec:definitions}

Let $R$ be a commutative principal ideal domain (PID) with $1$.
We assume that $R$ is computable, in the sense that all arithmetic operations, including gcd-computations and the computation of Bézout coefficients, can be performed effectively.
Classical examples of such rings are the integers $\ZZ$ or the univariate polynomial ring $K[x]$ over a field $K$, with the extended Euclidean algorithm.

We denote by $A=\RXY$ the \emph{mixed algebra} of polynomials in commutative (or central) variables $x_{1},\dots,x_{k}$ and noncommutative variables $y_{1},\dots,y_{n}$.
Formally, this algebra is the quotient $\RX[X,Y]/(x_{i}y_{j}-y_{j}x_{i}, x_{i}x_{j}- x_{j}x_{i} \mid\forall i,j)$.
Note that, in this algebra, $x_{i}f = fx_{i}$ for all $f \in \RXY$, but in general $y_i f \neq f y_i$.

A (mixed) monomial in $A$ is a product $vw$ with $v = x_{1}^{\alpha_{1}}\cdots x_{k}^{\alpha_{k}}$ a (commutative) monomial in $X$ and $w = y_{i_{1}} y_{i_{2}} \cdots y_{i_{l}}$ a (noncommutative) word in $Y$.
We denote by $[X]$ the set of all commutative monomials in $X$ and by $\langle Y \rangle$ the set of all noncommutative words in $Y$.
A term in $A$ is the product of a nonzero coefficient in $R$ and a mixed monomial in $A$.
We denote by $M(A)$ the set of all mixed monomials of $A$, and by $T(A)$ the set of all terms of $A$.

Divisibility of mixed monomials and terms in $A$ is defined componentwise: given $c,d \in R$, $u,v \in [X]$ and $a,b \in \langle Y\rangle$, $cua$ divides $dvb$ if and only if $c$ divides $d$, $u$ divides $v$, and $a$ is a subword of $b$.

A monomial ordering $\leq$ is a total ordering on $M(A)$ which is compatible with multiplication, that is, $m \leq m'$ implies $amb \leq am'b$ for all $a,b,m, m' \in M(A)$,
and is such that every non-empty set has a minimal element (well-ordering).
A monomial ordering naturally defines a partial ordering on $T(A)$, called a term ordering.
For ease of notations, we extend the definition of monomial and term to contain $0$, which is assumed to be smaller than all other elements.

Given a family $(f_{1},\dots,f_{r}) \in A^r$ of polynomials in $A$, we let $I = (f_1,\dots,f_r)$ be the (two-sided) ideal generated by $f_1,\dots,f_r$.
Furthermore, we consider the free $A$-bimodule $\Sigma = (A \otimes_{R[X]} A)^{r}$, see also~\cite[Sec.~0.11]{cohn-free-rings-and-their-relations}.
We denote its canonical basis by $\varepsilon_{1},\dots,\varepsilon_{r}$, and we equip it with an $A$-bimodule homomorphism
$\overline{\cdot} : \varepsilon_{i} \mapsto f_{i}$.
A \emph{labeled polynomial} is a pair $(f,\alpha) \in I \times \Sigma$ with $f = \overline{\alpha}$, denoted $f^{[\alpha]}$.
The \emph{labeled module} generated by $f_1,\dots,f_r$ is the set of all labeled polynomials and denoted by $I\Sig$.
It is isomorphic to $\Sigma$ as an $A$-bimodule.

An element $\alpha \in \Sigma$ with $\overline \alpha = 0$ is called a \emph{syzygy of $I\Sig$}.
The set of all syzygies of $I\Sig$, denoted by $\Syz(I\Sig)$, forms an $A$\nobreakdash-bimodule. 

A (module) monomial in $\Sigma$ is a product $ua\varepsilon_{i}b$ where $u \in [X]$ and $a,b \in \<Y>$.
A term in $\Sigma$ is the product of a nonzero coefficient and a monomial.
The set of all monomials (resp.~terms) in $\Sigma$ is denoted by $M(\Sigma)$ (resp.~$T(\Sigma)$).
Every $\alpha \in \Sigma$ has a unique representation $\alpha = \sum_{i=1}^d c_i u_i a_i \varepsilon_{j_i} b_i$ with nonzero $c_i \in R$ and pairwise different $u_i a_i \varepsilon_{j_i} b_i \in M(\Sigma)$.

A module ordering is a total ordering on $M(\Sigma)$ which is compatible with multiplication and a well-ordering.
A module ordering is called \emph{fair} if for any monomial $\mu$, the set of all monomials that are smaller than $\mu$ is finite.
Given a monomial ordering $\leq$ on $A$ and a module ordering $\preceq$ on $\Sigma$, the orders are said to be compatible if for all $u,v \in [X]$, $a,b \in \langle Y \rangle$, 
$i \in \{1,\dots,r\}$ we have $ua < vb$ iff $ua \varepsilon_{i} \prec vb \varepsilon_{i}$ iff $u \varepsilon_{i} a \prec v \varepsilon_{i} b$.


As for polynomials, we extend module orderings to partial orderings on terms.
More precisely, for $\mu,\sigma \in M(\Sigma)$ and $c,d \neq 0 \in R$, we write 
$c\mu \simeq d\sigma$ if $\mu = \sigma$, and $c \mu \preceq d \sigma$ if $\mu \prec \sigma$ or $c\mu \simeq d \sigma$.

Given a monomial ordering, the leading term $\lt(f)$ of $f\in A$ is the largest term appearing in the support of $f$.
The leading monomial $\lm(f)$ and the leading coefficient $\lc(f)$ are the corresponding monomial and coefficient, respectively.
Given a module ordering, the signature $\sig(\alpha)$ of a module element $\alpha$ is the largest term appearing in the support of $\alpha$.
Note that signatures include coefficients and that they are compatible with scalar multiplication, i.e., $\sig(t \alpha t') = t \sig(\alpha) t'$ for $t,t' \in T(A)$.

%

\section{SIGNATURE GR\"OBNER BASES}
\label{sec:signature-basis}

We fix $(f_{1},\dots,f_{r}) \in A^r$ and let $I\Sig$ be the labeled module generated by $f_1,\dots,f_r$.
As in other settings, signature Gr\"obner bases in $A$ are characterised by the fact that all elements in $I\Sig$ are reducible in a way that does not increase the signature.
Such reductions are called \emph{$\s$-reductions} and are defined below.
In the setting of coefficient rings several notions of reductions exist (weak, strong, and also modular reductions by the coefficients).
In this work, we focus on strong reductions requiring divisibility of the leading coefficients.

\begin{definition}\label{def:reduction}
Let $f^{[\alpha]}, g^{[\beta]} \in I^{[\Sigma]}$.
A combination $h^{[\gamma]} = f^{[\alpha]} - t g^{[\beta]} b$ with $t \in T(A)$, $b \in \<Y>$ is a \emph{(top) $\sig$\nobreakdash-reduction} if 
$\lt(t g b) = \lt(f) > \lt(h)$ and $\sig(t \beta b) \preceq \sig(\alpha)$.
The $\sig$-reduction is \emph{regular} if $\sig(t \beta b) \prec \sig(\alpha)$, and \emph{singular}
if $\sig(t \beta b) = \sig(\alpha)$.
If such a combination exists, $f^{[\alpha]}$ is called \emph{(regular/singular) $\sig$-reducible} by $g^{[\beta]}$.
\end{definition}

As usual, an element $f^{[\alpha]}$ is (regular/singular) $\sig$\nobreakdash-reducible by a set $G\Sig \subseteq I\Sig$ if there exists 
$g^{[\beta]} \in G\Sig$ such that $f^{[\alpha]}$ is (regular/singular) $\sig$-reducible by $g^{[\beta]}$.
If the result of a $\sig$\nobreakdash-reduction of $f^{[\alpha]}$ is $0^{[\gamma]}$, then we say that $f^{[\alpha]}$ \emph{$\sig$-reduces to $0$}.
We extend these definitions also to sequences of $\sig$-reductions.

Note that $\sig$-reductions do not increase the signature, that is, $\sig(\gamma) \preceq \sig(\alpha)$, with equality for regular $\sig$-reductions and strict inequality 
for singular $\sig$-reductions.
There are also $\sig$-reductions which are neither regular nor singular; in this case $\sig(\gamma) \simeq \sig(\alpha)$.

We extend the definition of (strong) signature Gr\"obner bases from the commutative case~\cite[Def.~2.5]{francis-verron} in a straightforward way.

\begin{definition}\label{def:sig-GB}
A set $G\Sig \subseteq I\Sig$ is a \emph{(strong) signature Gr\"obner basis of $I\Sig$ (up to signature $\sigma \in T(\Sigma)$)} 
if all $f \in I^{[\Sigma]}$ with $f \neq 0$ (and $\sig(\alpha) \prec \sigma$) are $\sig$-reducible by $G\Sig$.
\end{definition}

By disregarding the module labelling from Definition~\ref{def:sig-GB}, one recovers the definition of classical Gr\"obner bases in $A$~\cite{mikhalev}.
The polynomial parts of a signature Gr\"obner basis of $I\Sig$ form a Gr\"obner basis of $I = (f_1,\dots,f_r)$.

Signature Gr\"obner bases were introduced with the goal of using signatures of (known) syzygies to predict reductions to zero,
and thereby, speed up Gr\"obner basis computations.
Later, it was observed that signature-based algorithms also allow to compute, as a byproduct, a Gr\"obner basis of the syzygy module as defined below.

\begin{definition}
An element $\alpha \in \Sigma$ is \emph{(top) reducible} by a set $H \subseteq \Sigma$ if there exist
$\gamma \in H$, $t \in T(A)$, $b \in \<Y>$ such that $\sig(\alpha) = \sig(t \gamma b)$.

A set of syzygies $H \subseteq \Syz(I\Sig)$ is a \emph{syzygy basis of $I\Sig$ (up to signature $\sigma \in T(\Sigma)$)}
if every nonzero syzygy $\alpha \in \Syz(I\Sig)$ (with $\sig(\alpha) \prec \sigma$) is reducible by $H$.
\end{definition}

\section{Ambiguities}
\label{sec:ambiguities}

In order to define S- (and~G\nobreakdash-)polynomials in our setting, we first adapt the notion of \emph{ambiguities} from~\cite{ber78} to our setting.
Ambiguities characterise situations where one term can be reduced in two different ways,
and the aim of a (signature) Gröbner basis computation is to resolve ambiguities by forming S- (and~G\nobreakdash-)polynomials.
We first recall the definition of ambiguities for noncommutative words, and then subsequently extend them to mixed monomials and to labeled polynomials.

\begin{definition}
Let $p,q \in \<Y>$.
If there exist $a,b \in \<Y> \setminus \{1\}$ with $|a| < |q|$ and $|b| < |p|$ such that $ap = qb$, resp.~$pa = bq$, then we call the tuple
$(a \otimes 1,1 \otimes b, p, q)$, resp.~ $(1 \otimes a, b\otimes 1, p,q)$, an \emph{overlap ambiguity} of $p$ and $q$.

If there exist $a,b \in \<Y>$ such that $p = aqb$, resp.~$apb = q$, then we call the tuple
$(1 \otimes 1,a \otimes b, p, q)$, resp.~$(a \otimes b, 1 \otimes 1, p, q)$, an \emph{inclusion ambiguity} of $p$ and $q$.

Finally, for every $m \in \<Y>$, we call the tuples
$(1 \otimes mq, pm \otimes 1, p,q)$ and $(qm \otimes 1, 1 \otimes mp, p,q)$ \emph{external ambiguities} of $p$ and $q$.
\end{definition}

\begin{remark}
Note that two words $p,q$ can only have finitely many overlap and inclusion ambiguities, but they always have infinitely many external ambiguities.
\end{remark}

\begin{definition}
Let $up$, $vq \in M(A)$ and $(a \otimes b, c \otimes d,p,q)$ be an ambiguity of $p$ and $q$.
An \emph{ambiguity} of $up$ and $vq$ is given by 
	$(u' a \otimes b, v' c \otimes d, up, vq)$,
where $u' = \lcm(u, v) / u$ and $v' = \lcm(u,v) / v$.

For $f,g \in A \setminus \{0\}$, an \emph{ambiguity} of $f$ and $g$ is $(a\otimes b, c\otimes d, f, g)$ where $(a \otimes b, c \otimes d, \lm(f), \lm(g))$ is an ambiguity of $\lm(f)$ and $\lm(g)$.
We denote by $\amb(f,g)$ the set of all ambiguities of $f$ and $g$.
When clear by the context, we shall simply write $(a \otimes b, c \otimes d) \in \amb(f,g)$.

We define analogously ambiguities of labeled polynomials by considering their polynomial parts.
\end{definition}

Note that if $(m_1\otimes n_1, m_2 \otimes n_2, f,g)$ is an ambiguity of $f$ and $g$, then $\lm(m_1 f n_1) = \lm(m_2 g n_2)$.


Before defining S- and G-polynomials with module labelling, we introduce some useful terminology for ambiguities of labeled polynomials.
In the following, for a pair of nonzero $f,g \in A$, let $\lcmlc(f,g)$ be the least common multiple of $\lc(f)$ and $\lc(g)$.

\begin{definition}
Let $f^{[\alpha]},g^{[\beta]} \in I\Sig$ be such that $f,g \neq 0$ and let $a = (m_1\otimes n_1, m_2 \otimes n_2, f^{[\alpha]},g^{[\beta]}) \in \amb(f^{[\alpha]},g^{[\beta]})$.
The \emph{leading monomial} of $a$ is
	$\lma(a) \coloneqq \lm(m_1 f n_1) = \lm(m_2 g n_2)$
and, with $c_f = \frac{\lcmlc(f,g)}{\lc(f)}$, $c_g = \frac{\lcmlc(f,g)}{\lc(g)}$, the \emph{signature} of $a$ is
\[
	\siga(a) \coloneqq \max\left(\sig(c_f m_1 \alpha n_1), - \sig(c_g m_2 \beta n_2)\right),
\]
choosing the first in case of tie.
The ambiguity $a$ is called \emph{regular} if $\sig(m_1 \alpha n_1) \not\simeq \sig(m_2 \beta n_2)$ and \emph{singular} if $\sig(c_f m_1 \alpha n_1) = \sig(c_g m_2 \beta n_2)$.
\end{definition}

\begin{remark}
Recall that the ordering on the signatures is only partial, and thus, it can happen that an ambiguity $a$ is neither regular nor singular.
\end{remark}

The following lemma asserts that any situation where multiples of two labeled polynomials share a common leading monomial but differ in their signatures,
can be characterised by a regular ambiguity of the two elements.

\begin{lemma}\label{lemma:ambiguities}
Let $g_1^{[\beta_1]}, g_2^{[\beta_2]}\in I\Sig$ be such that $g_1,g_2 \neq 0$ and let $t_i \in T(A), b_i \in \<Y>$, $i = 1,2$, such that
\[
	\lm(t_1 g_1 b_1) = \lm(t_2 g_2 b_2) \; \text{ and } \; \sig(t_1 \beta_1 b_1) \succ \sig(t_2 \beta_2 b_2).
\]
Then there exists a regular ambiguity $a \in \amb(g_1^{[\beta_1]}, g_2^{[\beta_2]})$, $t_3 \in T(A), b_3 \in \<Y>$ such that
	$t_3 \lma(a) b_3 = \lm(t_i g_i b_i)$ and $t_3 \siga(a) b_3 \simeq \sig(t_1 \beta_1 b_1)$,
with equality of signatures if $\lc(t_1 g_1 b_1) = \lc(t_2 g_2 b_2)$.
\end{lemma}
\begin{proof}
  Write $t_i = c_i u_i a_i$ and $\lm(g_i) = v_i w_i$ with $c_i \in R, u_i,v_i \in [X], a_i,w_i \in \<Y>$.
  Then, by assumption $u_1 v_1 = u_2 v_2$ and $a_1 w_1 b_1 = a_2 w_2 b_2 =: W$.
  If, in $W$, either of $w_1$ and $w_2$ is completely contained in the other, then there exists an inclusion ambiguity of $w_1$ and $w_2$ characterising this situation, 
  otherwise one of them starts earlier in $W$ and the other finishes later, in which case there is an overlap or external ambiguity of $w_1$ and $w_2$ characterising this situation.
  Hence, in any case, there exists an ambiguity $(p_1 \otimes q_1, p_{2} \otimes q_2)$ of $w_1$ and $w_2$ such that $p_1 w_1 q_1 = p_2 w_2 q_2$ is a subword of $W$,
  i.e., there exist $l,r \in \<Y>$ such that $l p_1 w_1 q_1 r = l p_2 w_2 q_2 r = W$.
  By definition $a = (v_1' p_1 \otimes q_1, v_2' p_2 \otimes q_2)$, with $v_i' = \lcm(v_1, v_2) / v_i$, is an ambiguity in $\amb(g_1^{[\beta_1]}, g_2^{[\beta_2]})$.
  We claim that $a$ satisfies the conditions of the lemma with $t_3 = m_3 l \in T(A)$, where $m_3 = u_1 v_1 / \lcm(v_1,v_2) =  u_2 v_2 / \lcm(v_1,v_2)$, and $b_3 = r$.
  The condition on the leading monomials is clear by construction. 
  For the claim concerning the signatures, we note that $m_3 v_i' = u_i$ and by choice of $l,r$ we have $l p_i = a_i$ and $q_i r = b_i$.
  Thus, for $i = 1,2$,
  \begin{align*}
      \sig(t_3 v_i' p_i \beta_i q_i b_3) &= \sig(m_3 v_i' l p_i \beta_i q_i r) \\
      &= \sig(u_i a_i \beta_i b_i) \simeq \sig(t_i \beta_i b_i),
 \end{align*}
 showing that $a$ is regular since $\sig(t_1 \beta_1 b_1) \succ \sig(t_2 \beta_2 b_2)$ and that 
 $t_3 \siga(a) b_3 = \sig(t_3 v_1' p_1 \beta_1 q_1 b_3) \simeq \sig(t_1 \beta_1 b_1)$.
  
  For the final part, assume that also $\lc(t_1 g_1 b_1) = \lc(t_2 g_2 b_2)$, i.e., $c_1 \lc(g_1) = c_2 \lc(g_2)$.
  Then, multiplying $t_3$ by $c_1 \lc(g_1) / \lcmlc(g_1,g_2)$ yields the claimed equality of signatures.
 \end{proof}

We extend the definition of S-polynomials to our setting.

\begin{definition}
Let $f^{[\alpha]},g^{[\beta]} \in I\Sig$ be such that $f,g \neq 0$ and let $a = (m_1\otimes n_1, m_2 \otimes n_2) \in \amb(f^{[\alpha]},g^{[\beta]})$.
The \emph{S-polynomial} of $a$ is
\[
	\SPol(a) \coloneqq \frac{\lcmlc(f,g)}{\lc(f)} m_1 f^{[\alpha]} n_1 - \frac{\lcmlc(f,g)}{\lc(g)} m_2 g^{[\beta]} n_2.
\]
\end{definition}

As usual, S-polynomials are defined so that leading terms cancel, i.e., if $h^{[\gamma]} = \SPol(a)$, then $\lm(h) \prec \lma(a)$.
Furthermore, $\sig(\gamma) \preceq \siga(a)$, with equality if and only if the ambiguity is regular
and strict inequality if and only if the ambiguity is singular.

In the setting of coefficient rings, we also need G-polynomials,
which do not aim at canceling leading terms but at obtaining minimal leading coefficients.

\begin{definition}
Let $f^{[\alpha]},g^{[\beta]} \in I\Sig$ be such that $f,g \neq 0$ and let $a = (m_1\otimes n_1, m_2 \otimes n_2) \in \amb(f^{[\alpha]},g^{[\beta]})$.
Furthermore, let $c,d \in R$ be Bézout coefficients of $\lc(f)$ and $\lc(g)$, \ie, $c \lc(f) + d \lc(g) = \gcd(\lc(f),\lc(g))$.
The \emph{G-polynomial} of $a$ w.r.t.\ $c,d$ is
\[
	\GPol_{c,d}(a) \coloneqq c m_1 f^{[\alpha]} n_1 + d m_2 g^{[\beta]} n_2.
\]
\end{definition}

While the coefficients $c,d$ in the definition of G-polynomials are not unique, the leading term is.
More precisely, the leading monomial of $\GPol_{c,d}(a)$ is $\lma(a)$ and the leading coefficient is $\gcd(\lc(f),\lc(g))$.
The signature of the G\nobreakdash-polynomial, however, depends on the choice of $c,d$.
A crucial observation \cite[Prop.~2.14]{francis-verron} is that these coefficients can be chosen so that G-polynomials are \emph{never} singular, i.e., 
so that the signatures of the two summands do not cancel each other.


\begin{lemma}\label{lemma:g-pol-non-singular}
Let $f^{[\alpha]},g^{[\beta]} \in I\Sig$ and let $a \in \amb(f^{[\alpha]},g^{[\beta]})$.
There exist $c,d \in R$ such that $\sig(\GPol_{c,d}(a)) \simeq \siga(a)$.
\end{lemma}

\begin{proof}
The proof of~\cite[Prop.~2.14]{francis-verron} only relies on properties of the leading coefficients and carries over directly to our setting.
\end{proof} 

We only consider Bézout coefficients as in Lemma~\ref{lemma:g-pol-non-singular}, and refer to the corresponding G-polynomial
as \emph{the} G-polynomial of $a$, denoted by $\GPol(a)$.
Note that if $h^{[\gamma]} = \GPol(a)$, then $\sig(\gamma) \simeq \siga(a)$ and $\lm(h) = \lma(a)$.

The following lemma captures the significance of G-polynomials for our computations.
It states that if a leading term can be written as a sum of two other leading terms, then it is divisible by the leading term of a G-polynomial.

\begin{lemma}\label{lemma:g-pol}
Let $f^{[\alpha]}, g_1^{[\beta_1]}, g_2^{[\beta_2]} \in I\Sig$ be such that $f,g_1,g_2 \neq 0$ and such that there exist $t_i \in T(A), b_i \in \<Y>$, $i = 1,2$, with
\[
	\lt(f) = \lt(t_1 g_1 b_1) + \lt(t_2 g_2 b_2) \; \text{ and } \; \sig(t_1 \beta_1 b_1) \succ \sig(t_2 \beta_2 b_2).
\]
Then there exist $g_3^{[\beta_3]} = \GPol(a)$ for some $a \in \amb(g_1^{[\beta_1]}, g_2^{[\beta_2]})$ and $t_3 \in T(A), b_3 \in \<Y>$ such that
	$\lt(t_3 g_3 b_3) = \lt(f)$ and $\sig(t_3 \beta_3 b_3) \simeq \sig(t_1 \beta_1 b_1)$.
\end{lemma}
\begin{proof}
Note that $\lm(t_1 g_1 b_1) = \lm(t_2 g_2 b_2)$.
Thus, the result follows from Lemma~\ref{lemma:ambiguities},~\ref{lemma:g-pol-non-singular}, and the properties of G-polynomials.
\end{proof}

\begin{definition}\label{def:complete}
A set $G\Sig \subseteq I\Sig$ is \emph{complete} if the G-poly-nomials of all ambiguities of $G\Sig$ are $\sig$-reducible by $G\Sig$. 
\end{definition}

A set can be completed by adding G-polynomials to it.
The following definition extends the idea of G-polynomials to syzygies.
To this end, we define the least common multiple of two module monomials as follows.
Let $\sigma_i = u_i a_i \varepsilon_j b_i \in M(\Sigma)$, $i = 1,2$.
If $a_{k'}$ is a suffix of $a_k$ and $b_{l'}$ is a prefix of $b_l$, where $\{k,k'\} = \{l,l'\} = \{1,2\}$, 
then $\lcm(\sigma_1, \sigma_2) \coloneqq \lcm(u_1,u_2) a_k \varepsilon_{j} b_l$.

\begin{definition}
Let $\gamma_1, \gamma_2 \in \Syz(I\Sig)$ be such that $\sig(\gamma_i) = c_i \sigma_i$ with $c_i \in R$, $\sigma_i \in M(\Sigma)$. 
Assume that $\lcm(\sigma_1,\sigma_2)$ is defined and let $m_i \in M(A)$, $b_i \in \<Y>$ be such that $\lcm(\sigma_1,\sigma_2) = m_i \sigma_i b_i$.
Also, let $c,d$ be Bézout coefficients of $\gcd(c_1, c_2)$.
The \emph{sig-Combination} of $\gamma_1$ and $\gamma_2$ is
	$\sigComb(\gamma_1,\gamma_2) \coloneqq c m_1 \gamma_1 b_1 + d m_2 \gamma_2 b_2$.
	
A set $H \subseteq \Syz(I\Sig)$ is \emph{sig-complete} if any sig-Combination of elements in $H$ is reducible by $H$.
\end{definition}

To end this section, we introduce a concept needed later.

\begin{definition}
Let $f^{[\alpha]} \in I\Sig$ and $G\Sig \subseteq I\Sig$.
We say that $f^{[\alpha]}$ is \emph{super reducible} by $G\Sig$ if there exist $g^{[\beta]} \in G\Sig$ and $t \in T(A)$, $b \in \<Y>$ such that
$\sig(\alpha) = \sig(t \beta b)$ and $\lm(f) = \lm(t g b)$.
\end{definition}

Note that super reducibility need not imply $\sig$-reducibility as the former only requires equality of the leading monomials, without considering the leading coefficients.
However, if a set of reducers is complete, a super reducible element is also $\sig$-reducible.
This the result of Proposition~\ref{prop:complete}, which is an adaptation of~\cite[Prop.~2.19]{francis-verron}.

\begin{proposition}\label{prop:complete}
Let $f^{[\alpha]} \in I\Sig$ and $G\Sig \subseteq I\Sig$ be complete and a signature Gr\"obner basis up to signature $\sig(\alpha)$.
If $f^{[\alpha]}$ is super reducible by $G\Sig$, then it is also $\s$-reducible by $G\Sig$.
\end{proposition}
\begin{proof}
We essentially follow the proof of~\cite[Prop.~2.19]{francis-verron}.
Super reducibility implies the existence of $g_1^{[\beta_1]} \in G\Sig$ and $t_1 \in T(A)$, $b_1 \in \<Y>$ such that
$\sig(\alpha) = \sig(t_1 \beta_1 b_1)$ and $\lm(f) = \lm(t_1 g_1 b_1)$.
If, in fact, $\lt(f) = \lt(t_1 g_1 b_1)$, then $f^{[\alpha]}$ is $\sig$-reducible by $g_1^{[\beta_1]}$.
Otherwise, with $h^{[\gamma]} = f^{[\alpha]} - t_1 g_1^{[\beta_1]} b_1$, we have $\lm(h) = \lm(f)$ and $\sig(\gamma) \prec \sig(\alpha)$.
By assumption, $h^{[\gamma]}$ is $\sig$-reducible by $G\Sig$. 
Let $g_2^{[\beta_2]}$ be such a reducer with $t_2 \in T(A)$, $b_2 \in \<Y>$ such that $\sig(\gamma) \succeq \sig(t_2 \beta_2 b_2)$ and $\lt(h) = \lt(t_2 g_2 b_2)$.
Consequently, we have
$\lt(f) = \lt(t_1 g_1 b_1) + \lt(t_2 g_2 b_2) \; \text{ and } \; \sig(\alpha) \succ \sig(\gamma) \succeq \sig(t_2 \beta_2 b_2)$.
By Lemma~\ref{lemma:g-pol} there exists $g_3^{[\beta_3]} = \GPol(a)$ for some $a \in \amb(g_1^{[\beta_1]},g_2^{[\beta_2]})$ and $t_3 \in T(A), b_3 \in \<Y>$ such that
$\lt(t_3 g_3 b_3) = \lt(f) \; \text{ and } \; \sig(t_3 \beta_3 b_3) \simeq \sig(t_1 g_1 b_1) = \sig(\alpha)$.
Since $G\Sig$ is complete, $g_3^{[\beta_3]}$ is $\sig$-reducible by $G\Sig$, and any reducer of $g_3^{[\beta_3]}$ can be used to $\sig$-reduce $f^{[\alpha]}$.
\end{proof}

\section{Cover criterion}
\label{sec:cover}

In this section, we adapt the cover criterion~\cite{GVW,francis-verron} to our setting, yielding an effective characterisation of signature Gröbner bases.

\begin{definition}\label{def:cover}
Let $G\Sig \subseteq I\Sig$ and $H \subseteq \Syz(I\Sig)$.
Furthermore, let $g_1^{[\beta_1]}, g_2^{[\beta_2]} \in I\Sig$.
An ambiguity $a \in \amb(g_1^{[\beta_1]}, g_2^{[\beta_2]})$ is \emph{covered} by $(G\Sig,H)$ if there exist $g^{[\beta]} \in G\Sig$, 
$\gamma \in H$ and $t,t' \in T(A)$, $b,b'\in \<Y>$ such that the following conditions hold:
\begin{itemize}
	\item $\siga(a) = \sig(t \beta b) + \sig(t' \gamma b')$;
	\item $\lm(t g b) < \lma(a)$;
\end{itemize}
\end{definition}

\begin{remark}
In Definition~\ref{def:cover}, either of $t$ and $t'$ can also be $0$.
If $t = 0$, then $\lm(tgb) = 0$ and the second condition is trivially fulfilled.
\end{remark}

Our version of the cover criterion differs in one main point from the classical one~\cite[Thm.~2.4]{GVW} for (commutative) polynomials over coefficient
fields: We need to consider linear combinations to form the signature.
This requirement comes from the fact that we deal with coefficient rings and is also necessary in the commutative case, see~\cite[Def.~2.20]{francis-verron}.

We can prove a characterisation of noncommutative signature Gr\"obner bases and syzygy bases using the cover criterion.
The following theorem is an adaptation of~\cite[Thm.~3.1]{francis-verron}.

\begin{theorem}\label{thm:correct}
  Let $G\Sig \subseteq I\Sig$ and $H \subseteq \Syz(I\Sig)$ be such that the following conditions hold:
  \begin{enumerate}
    \item for all $g^{[\beta]} \in G\Sig : g \neq 0$;
    \item $G\Sig$ is complete and $H$ is sig-complete;
    \item all $\sigma \in T(A)$ are reducible by $H \cup \{\beta \mid g^{[\beta]} \in G\Sig \}$;
    \item all regular ambiguities of $G\Sig$ are covered by $(G\Sig,H)$;
  \end{enumerate} 
  Then $G\Sig$ is a signature Gr\"obner basis and $H$ is a syzygy basis of $I\Sig$.
\end{theorem}

\begin{proof}
The proof is an adaption of that of~\cite[Thm.~2.4]{GVW}.
Assume, for contradiction, that there exists $f^{[\alpha]} \in I\Sig$ such that either $f \neq 0$ and $f^{[\alpha]}$ is not $\s$-reducible by $G\Sig$ or 
$f = 0$ and $\alpha$ is not reducible by $H$.
Pick such $f^{[\alpha]}$ with minimal signature.
Note that this implies that $G\Sig$ is, by definition, a signature Gr\"obner basis up to signature $\sig(\alpha)$.

Let $g_1^{[\beta_1]} \in G\Sig$, $\gamma_1 \in H$ and $t_1,t_1' \in T(A)$, $b_1,b_1'\in \<Y>$ such that
\begin{equation}\label{eq:cover-1}
	\sig(\alpha) = \sig(t_1 \beta_1 b_1) + \sig(t_1' \gamma_1 b_1').
\end{equation}
By assumption, such a decomposition exists (in fact, with either $t_1 = 0$ or $t_1' = 0$ but we do not require this for~\eqref{eq:cover-1}).
We select these elements so that $\lm(t_1g_1b_1)$ is minimal and claim that $t_1g_1^{[\beta_1]}b_1$ is not regular $\s$-reducible by $G\Sig$.

To prove this, suppose that $t_1g_1^{[\beta_1]}b_1$ is regular $\s$-reducible by $g_2^{[\beta_2]}$, i.e., there exist $t_2 \in T(A), b_2 \in \<Y>$ such that 
$\lt(t_1 g_1 b_1) = \lt(t_2 g_2 b_2)$ and $\s(t_1 \beta_1 b_1) \succ \s(t_2 \beta_2 b_2)$.
Then Lemma~\ref{lemma:ambiguities} implies the existence of a regular ambiguity $a \in \amb(g_1^{[\beta_1]}, g_2^{[\beta_2]})$ and $t_3 \in T(A), b_3 \in \<Y>$ such that
\begin{equation}\label{eq:cover-2}
	t_3 \lma(a) b_3 = \lm(t_i g_i b_i) \; \&  \; t_3 \siga(a) b_3 = \sig(t_1 \beta_1 b_1).
\end{equation}
By assumption $a$ is covered by $(G\Sig,H)$, i.e., there exist
$g^{[\beta]} \in G\Sig$, $\gamma \in H$ and $t,t' \in T(A)$, $b,b'\in \<Y>$ such that
$\lm(t g b) < \lma(a)$ and $\siga(a) = \sig(t \beta b) + \sig(t' \gamma b')$.
With~\eqref{eq:cover-1} and~\eqref{eq:cover-2}, this yields
$$\sig(\alpha) = \sig(t_3 t \beta b b_3) + \sig(t_3 t' \gamma b' b_3) + \sig(t_1' \gamma_1 b_1')$$
and $\lm(t_3 t g b b_3) < \lm(t_1 g_1 b_1)$.
Let $\delta = \sigComb(\gamma_1, \gamma) \in H$, which is well-defined.
Since, by definition, $\sig(\delta)$ divides the sum $\sig(t_3 t' \gamma b' b_3) + \sig(t_1' \gamma_1 b_1')$,
the pair $(g^{[\beta]}, \delta)$ yields a decomposition of $\sig(\alpha)$ with smaller leading monomial than $\lm(t_1 g_1 b_1)$; a contradiction to the minimality of $\lm(t_1 g_1 b_1)$.

Thus, $t_1g_1^{[\beta_1]}b_1$ is not regular $\s$-reducible.
Now, we distinguish between two cases depending on whether $f = 0$ or not.

If $f \neq 0$, then $\lm(f) \neq \lm(t_1 g_1 b_1)$ because otherwise $f^{[\alpha]} -  t_1' 0^{[\gamma_1]} b_1'$ would 
be super reducible by $g_1^{[\beta_1]}$, and thus, by Proposition~\ref{prop:complete}, $\sig$-reducible by $G\Sig$.
But then also $f^{[\alpha]}$ would be $\sig$-reducible -- a contradiction.

Let $f_1^{[\alpha_1]} = f^{[\alpha]} - t_1 g_1^{[\beta_1]} b_1 - t_1' 0^{[\gamma_1]} b_1'$.
Then $\sig(\alpha_1) \prec \sig(\alpha)$ and $\lt(f_1) = \max(\lt(f), \lt(t_1g_1b_1)) \neq 0$.
By minimality of $\sig(\alpha)$,  $f_1^{[\alpha_1]}$ is $\s$-reducible by $G\Sig$.
But this implies that $f^{[\alpha]}$ or $t_1 g_1^{[\beta_1]} b_1$ is regular $\s$-reducible by $G\Sig$, which is a contradiction.

If $f = 0$ and $f_1^{[\alpha_1]}$ is like before, then $\lt(f_1) = \lt(t_1 g_1 b_1)$.
Now, if $f_1 \neq 0$, then $f_1^{[\alpha_1]}$ being $\s$-reducible implies that $t_1 g_1^{[\beta_1]} b_1$ is regular $\s$-reducible by $G\Sig$, which is a contradiction.
Thus $f_1 = t_1 g_1 b_1 = 0$, and by assumption on $G\Sig$, we can conclude $t_1 = 0$, showing that $\alpha$ is reducible by $\gamma_1 \in H$.
\end{proof}

\vspace{-0.3cm}
\section{Algorithm}
\label{sec:algorithm}

\subsection{Description}
\label{sec:description}

\vspace{-0.3cm}
\begin{algorithm}[h]
  \SetAlgoLined
  \KwIn{$(f_1,\dots,f_r) \in A^{r}$ generating a labeled module $I^{[\Sigma]}$}
  \KwOut{
    \par
    \begin{minipage}[t]{0.9\linewidth}
      \vspace{-0.25cm}
      \begin{itemize}[leftmargin=3em]
        \setlength{\itemindent}{-15pt}
        \item $G^{[\Sigma]}\subseteq I\Sig$ a signature Gr\"obner basis of $I^{[\Sigma]}$
        \item $H \subseteq \Syz(I^{[\Sigma]})$ a syzygy basis of $I^{[\Sigma]}$
      \end{itemize}
    \end{minipage}
  }
  $G^{[\Sigma]} \leftarrow \emptyset$ ;
  $H \leftarrow \emptyset$ ;
  $P \leftarrow \{(f_1^{[\varepsilon_1]},\mathtt{N}),\dots,(f_r^{[\varepsilon_r]},\mathtt{N})\}$ \;
  \While{$P \neq \emptyset$}{
   select and remove $(p^{[\pi]},\mathtt{type})$ from $P$\label{line choice p 1} \label{algoline:sig-select}\;
    \uIf{\texttt{type} is not $\mathsf{S}(a)$ or $a$ is not covered by $(G^{[\Sigma]},H)$} {
      $p'^{[\pi']} \leftarrow$ result of regular $\s$-reducing $p^{[\pi]}$ by $G^{[\Sigma]}$ \label{line reduction sigGB 1} \;
      \uIf{$p' = 0$}{
        $H \leftarrow H \cup \{\pi'\}$, and make $H$ sig-complete \;}
      \uElse{
        $G^{[\Sigma]} \leftarrow G^{[\Sigma]} \cup \{p'^{[\pi']}\}$\;
 
        \For{$g^{[\beta]} \in G^{[\Sigma]}$ and $a \in \amb(p'^{[\pi']}, g^{[\beta]})$\label{algoline:sig-begin-add}}
         {
           {
            add $(\SPol(a),\mathsf{S}(a))$ to $P$ if $a$ is regular\;}
              add $(\GPol(a)),\mathsf{G}(a))$ to $P$\;
           \label{algoline:sig-end-add}
           }
        }
      }
    }
  \Return $G^{[\Sigma]},H$
  \caption{Signature Gr\"obner basis}\label{algo:labeledGB}
\end{algorithm}
\vspace{-0.3cm}


Equipped with Theorem~\ref{thm:correct}, we can describe an algorithm for enumerating signature Gröbner bases and syzygy bases in $A$.
This algorithm combines key elements of Kandri-Rody and Kapur's algorithm~\cite{Kandri-Rody-Kapur} for Gröbner bases over $\ZZ$ with signatures~\cite{francis-verron} and the general algorithm for computing signature Gröbner bases in free algebras over fields~\cite{hofstadler-verron}.
The difference between the free algebra and the mixed algebra is not apparent in the algorithm, but the computation of ambiguities, S- and G-polynomials and the reductions are using the definitions from the mixed algebra.

The algorithm runs on a loop, processing polynomials from a queue.
At each step, it selects a polynomial, and if not redundant, reduces it and forms new polynomials from it, adding them to the queue.
The algorithm ensures that, for every ambiguity of $G\Sig$, an element is eventually added which covers it.

As in the classical case, S- and G-polynomials are added for different purposes: S-polynomials create new leading monomials, and G-polynomials create new leading coefficients for existing leading monomials.
The cover criterion concerns the existence of leading monomials in the basis, and can be used to skip over S-polynomials which correspond to an ambiguity that is already covered.
The situation is different for G-polynomials: even if the corresponding ambiguity is covered, it is necessary to process and add the G-polynomial to complete the set $G\Sig$.

This requires keeping track of the construction of each labeled polynomial in the algorithm.
For this purpose, the main queue $P$ contains pairs $(p^{[\pi]},\mathtt{type})$, where $\mathtt{type}$ is either
 the symbol $\mathsf{N}$ indicating that $p$ is one of the input polynomials, or the symbol $\mathsf{S}(a)$, resp.~$\mathsf{G}(a)$, indicating that $p^{[\pi]}$ is the S-polynomial, resp.~the G-polynomial, of the ambiguity $a$.


As is classically done when presenting signature-based algorithms, the definitions, properties and the algorithm are stated using labeled polynomials, carrying their entire module representation.
However, all those properties only require considering the signature of the module representation, and the true strength is that the algorithms need only maintain pairs $(f,\sig(\alpha))$ instead of labeled polynomials $f^{[\alpha]}$.
In that case, the output only contains the signatures of the computed polynomials and syzygies, but the full module representations can be recovered \emph{a posteriori}. We refer to~\cite{GVW,hofstadler-verron} for details on that reconstruction step.

\subsection{Spooling the list of pairs \& Correctness}

For brevity, the presentation of the algorithm is simplified concerning the handling of the queue of pairs.
Unlike in the commutative case or the free case over fields, in general the inserting from line~\ref{algoline:sig-begin-add} to~\ref{algoline:sig-end-add} involves infinitely many elements.
In order for the algorithm to correctly enumerate a signature Gröbner basis, it needs to process all possible S- and G-polynomials, so it cannot enter an infinite loop inside of the main loop (necessarily infinite).

Instead, the algorithm should be modified to add a spooling machinery ensuring that all pairs are processed.
We detail a possible implementation of such a mechanism.
The main idea is that the algorithm must ensure that $P$ is finite at all times.
The spooling mechanism ensures that by adding a finite number of pairs to $P$, then adding more whenever they have all been processed.
The additional mechanism would then run in parallel for all sources of pairs, ensuring that the pairs are processed in a fair order.

More precisely, we assume that we know a way to enumerate ambiguities, that is, we are given two functions:
\begin{itemize}
  \item $\mathtt{firstamb}$ taking as input two labeled polynomials $f^{[\alpha]}$ and $g^{[\beta]}$, and returning an ambiguity $a_0$ of them;
  \item $\mathtt{nextamb}$ taking as input an ambiguity $a$ of two labeled polynomials $f^{[\alpha]}$ and $g^{[\beta]}$, and returning another ambiguity $n(a)$ of them;
\end{itemize}
with the property that $\{a_0,n(a_0),n(n(a_0)),\dots\} = \amb(f^{[\alpha]},g^{[\beta]})$.

The algorithm would maintain an additional variable \texttt{spool}, which is a finite list of tuples formed of two labeled polynomials and an ambiguity between them.
The lines~\ref{algoline:sig-begin-add} to~\ref{algoline:sig-end-add} would be replaced by initialization in that list:

\smallskip

\bgroup
\removelatexerror
\begin{algorithm}[H]
  \renewcommand{\thealgocf}{1a}
  \caption{Add the ambiguities to the spooling queue}
  \renewcommand{\theAlgoLine}{12\alph{AlgoLine}}
  \renewcommand{\thealgocf}{\arabic{algocf}}
  \SetAlgoLined
          \For{$g^{[\beta]} \in G^{[\Sigma]}$}{
    $a_0 \leftarrow \mathtt{firstamb}(p'^{[\pi']}, g^{[\beta]})$\;
    add $a_0$ to $\mathtt{spool}$\;
  }
\end{algorithm}
\egroup

\smallskip

The construction of the pairs would be done closer to their actual use, by adding the following lines at the very end of the main loop, just prior to selection of the next pair:

\smallskip

\bgroup
\removelatexerror
\begin{algorithm}[H]
 \renewcommand{\thealgocf}{1b}
  \caption{Process ambiguities in the queue}
  \renewcommand{\thealgocf}{\arabic{algocf}}
  \renewcommand{\theAlgoLine}{16\alph{AlgoLine}}
  \SetAlgoLined
  \def\idxcmd{i}
  \For{$\idxcmd \in \{1,\dots,\#\mathtt{spool}\}$}{
    $a \leftarrow \mathtt{spool}[\idxcmd]$\;
    {
      add $(\SPol(a),\mathsf{S}(a))$ to $P$ if $a$ is regular\;}
      add $(\GPol(a),\mathsf{G}(a))$ to $P$ \;
    $\mathtt{spool}[\idxcmd] \leftarrow \mathtt{nextamb}(a)$\;
  }
\end{algorithm}
\egroup

\smallskip

This has the effect of adding finitely many new elements to $P$, namely at most $2$ for each element of $\mathtt{spool}$, and updating $\mathtt{spool}$ to generate new elements the next time the code is evaluated.

The main property that the machinery must satisfy is that it should correctly enumerate all ambiguities of pairs of polynomials.
This property, called a fair selection strategy, was originally introduced for the computation of subalgebra bases~\cite{Kapur-1989-CompletionProcedureComputing, Ollivier} and adopted for noncommutative bases in~\cite{Mora-1994-IntroductionToCommutative}.

\begin{definition}
  Algorithm~\ref{algo:labeledGB} is called \emph{fair} if, given any ambiguity of elements in $G^{[\Sigma]}$,
  the corresponding S- and G-polynomials are eventually processed in the main loop or discarded.
\end{definition}

\begin{proposition}
  With the structure described above, and if insertion and selection in $P$ are done first in, first out, the algorithm is fair.
\end{proposition}
\begin{proof}
  Let $g_{i}^{[\beta_{i}]},g_{j}^{[\beta_{j}]} \in G^{[\Sigma]}$ and let $a$ be an ambiguity between them.
  Since the function $\mathtt{nextamb}$ enumerates all ambiguities of  $g_{i}^{[\beta_{i}]}$ and $g_{j}^{[\beta_{j}]}$, eventually the insertion mechanism reaches the ambiguity $a$, at which point the corresponding S- and G-polynomials are either inserted into $P$ or discarded.
  If the S-polynomial is discarded, there is nothing to prove.
  If it is inserted, let $N$ be the length of the list $P$ after insertion.
  Since selection in $P$ is done on a first in, first out basis, after $N$ runs through the loop, the S-polynomial will be processed.
  The same applies to the G-polynomial.
\end{proof}


\begin{theorem}
  If Algorithm~\ref{algo:labeledGB} is fair, then it correctly enumerates a signature Gröbner basis and a syzygy basis of $I^{[\Sigma]}$.
\end{theorem}
\begin{proof}
  We prove that the algorithm enforces the requirements of Theorem~\ref{thm:correct}.
  First, by construction all elements added to $G^{[\Sigma]}$ have a nonzero polynomial part.
  The bases $G\Sig$ and $H$ are complete and sig-complete respectively, because all G-polynomials  and all sig-Combinations are added to them.
  All signatures $\varepsilon_{i}$ are processed, and result in either an element in $H$ or in $G\Sig$, depending on whether $f_i^{[\varepsilon_{i}]}$ regular $\s$-reduces to $0$ or not.
  Finally, the algorithm ensures that all ambiguities it considers are covered: either the ambiguity $a$ is already covered, or an element is added to either $G^{[\Sigma]}$ or $H$ with signature $\siga(a)$.
  Either way, this element covers the ambiguity $a$.
  Finally, since the algorithm is fair, it processes all ambiguities, covering all of them eventually.
\end{proof}

\begin{remark}\label{remark:fair-module-ordering}
If the used module ordering is fair and the selection strategy processes elements by increasing signatures, then, whenever the algorithm considers a labeled polynomial $p^{[\pi]}$,
the sets $G^{[\Sigma]}$ and $H$ are a signature Gröbner basis and a syzygy basis respectively up to signature $\sig(\pi)$.
\end{remark}

\subsection{Elimination criteria}
\label{sec:elimination}

For simplicity, we presented the algorithm stripped to its main loop, with only the cover criterion necessary for the loop invariant.
In this section, we list a few additional criteria which can be added to the algorithm to skip redundant ambiguities and polynomials.
They all work similarly to their counterpart in the commutative case, ensuring that elements are covered (for S-polynomials) or $\sig$-reducible (for G-polynomials).
First, we focus on G-polynomials.

\begin{corollary}
  Let $g_{1}^{[\beta_{1}]},g_{2}^{[\beta_{2}]} \in G^{[\Sigma]}$ and $a \in \amb(g_1^{[\beta_{1}]},g_{2}^{[\beta_{2}]})$.
  If $\GPol(a)$ is $\sig$-reducible by $G^{[\Sigma]} $, then it can be discarded.
\end{corollary}
\begin{proof}
 Follows from Theorem~\ref{thm:correct} and Definition~\ref{def:complete}.
\end{proof}

The criterion allows to avoid all reductions of G-polynomials to zero,
and, in fact, since we only consider top reductions, it avoids reductions of G-polynomials entirely. 
As a special case, we see that, if $\lc(g_1) \divides \lc(g_2)$, then all their G-polynomials can be discarded as they are all $\sig$-reducible by $g_1^{[\beta_{1}]}$.


We now move to S-polynomials.
Excluding covered ambiguities encapsulates a number of criteria, including the \emph{Syzygy criterion}~\citep[e.g.][Lem.~6.1]{eder:2017:survey}, stating that any regular ambiguity whose signature is reducible by the signature of a syzygy in $H$ can be discarded, and the \emph{Singular criterion}~\citep[e.g.][Lem.~6.2]{eder:2017:survey}, which says that for each signature only one regular ambiguity (the one with minimal leading monomial) has to be considered.
To make the cover criterion even stronger, we can exploit the fact that signatures of some syzygies, so-called \emph{trivial syzygies}, can be predicted in advance without having to perform any reductions.
In particular, for all $g_1^{[\beta_1]},g_2^{[\beta_2]} \in I\Sig$ and $m \in M(A)$, we obtain a trivial syzygy $g_2 m \beta_1 - \beta_2 m g_1$.
The aim of the F5 criterion is to detect these trivial syzygies.

\begin{corollary}[F5 criterion]\label{cor:F5-criterion}
Let $a$ be a regular ambiguity of elements in $G\Sig$ such that there exist $g_{1}^{[\beta_{1}]} ,g_{2}^{[\beta_{2}]} \in G\Sig$ and
$m \in M(A)$ with 
\begin{itemize}
	\item $\sig(g_2 m \beta_1) \not\simeq \sig(\beta_2 m g_1)$, and 
	\item $\siga(a)$ is reducible by $\max(\sig(g_2 m \beta_1), -\sig(\beta_2 m g_1))$.
\end{itemize}
Then $a$ is covered by the trivial syzygy $g_2 m \beta_1 - \beta_2 m g_1$ and $\SPol(a)$ can be discarded after adding this trivial syzygy to the set $H$.
\end{corollary}

The F5 criterion as phrased above also includes Buchberger's coprime criterion for eliminating
S-polynomials coming from elements with coprime leading terms, see~\cite[Lem.~22]{LMAZ20} for a noncommutative version without signatures.
In particular, if $g_1^{[\beta_{1}]}, g_2^{[\beta_{2}]}$ are such that $\lc(g_1)$ and $\lc(g_2)$ as well as the commutative parts of $\lm(g_1)$ and $\lm(g_2)$ are coprime, then, for every regular external ambiguity of these elements, the S-polynomial can be discarded after adding a suitable trivial syzygy to $H$.

Unlike in the commutative case, it is not possible to simply add all trivial syzygies to $H$ whenever a new element is added to the signature Gröbner basis, as there are infinitely many.
However, in order to check the F5 criterion, only finitely many syzygies can apply, and we can construct them on demand.
This renders the complexity of applying Corollary~\ref{cor:F5-criterion} essentially quadratic in the size of $G\Sig$.
In Section~\ref{sec:homogenization}, we show that, analogous to the commutative case, this cost can be reduced to linear for homogeneous polynomials under certain module orderings.

It was observed in \cite{hofstadler-verron} that for some labeled modules $I^{[\Sigma]}$, there exist a finite set of labeled polynomials $G^{[\Sigma]}$, and a finite set of syzygies $H$, such that $G^{[\Sigma]}$ is a signature Gröbner basis of $I\Sig$, and $H \cup \{\text{trivial syzygies formed with elements of $G^{[\Sigma]}$}\}$ is a syzygy basis of $I^{[\Sigma]}$.
If that is the case, Algorithm~\ref{algo:labeledGB} will eventually have produced $G^{[\Sigma]}$ and $H$, and $P$ will only contain an infinite set of elements with signature all reducible by a trivial signature.

A particular case is that where $R$ is a field and there are no commutative variables, that is, when we are using Algorithm~\ref{algo:labeledGB} in the free algebra over a field.
In that setting, Buchberger's coprime criterion excludes all external ambiguities, which leads to two elements only having finitely many ambiguities,
and the algorithm may terminate.
The advantages of Algorithm~\ref{algo:labeledGB} over that in~\cite{hofstadler-verron} is the stronger cover criterion and the ability to use any module order (Section~\ref{sec:elimination-orders}).

\vspace{-0.1cm}
\section{Applications}
\label{sec:applications}

\subsection{Elimination orders}
\label{sec:elimination-orders}
\vspace{-0.1cm}

Write $X = X_1 \dot\cup X_{2}$ and $Y = Y_1 \dot\cup Y_{2}$.
Recall that a monomial ordering $\leq$ on $M(A)$ is an \emph{elimination ordering} for $X_1 \cup Y_1$ 
if $\lm(f) \in [X_2]\<Y_2>$ implies $f \in R[X_2]\<Y_2>$ for all $f \in A$.
If $G$ is a Gröbner basis of $I$ for such an ordering, then
 $G \cap R[X_2]\langle Y_2\rangle = \{g \in G \mid g \in R[X_2]\langle Y_2\rangle \}$
is a Gröbner basis of $I \cap R[X_2]\langle Y_2\rangle$, see~\cite{Borges_1998}.

Those orders are thus very useful when simplifying expressions (or in particular solving systems), and signatures allow to prove that the simplified expression is correct.
Unfortunately, in the case of the free algebra, any module ordering compatible with an elimination ordering cannot be fair.
As such, even if $X = \emptyset$ and $R$ is a field, the algorithm from \cite{hofstadler-verron} was not applicable.
On the other hand, provided that ambiguities are processed in a fair order, Algorithm~\ref{algo:labeledGB} correctly enumerates a signature Gröbner basis with respect to any order, including elimination orders.

\vspace{-0.1cm}
\subsection{Ideal arithmetic}
\label{sec:ideal-arithmetic}
\vspace{-0.1cm}

The next application we mention concerns the use of Gröbner bases to perform ideal arithmetic operations. 
We recall~\cite{Nordbeck_1998} that, given two ideals $I,J \subseteq R[X]\< Y >$, a Gröbner basis of their intersection is obtained by eliminating a new commutative indeterminate $t$ from $t I + (1-t)J \subseteq R[X,t]\< Y>$.
Similarly, computing the homogeneous part of an ideal $I$, that is, computing a Gröbner basis of the subideal of $I$ generated by all homogeneous polynomials in $I$, can be realised by 
a computation in $R[X,T,T^{-1}]\< Y >$, where $T,T^{-1}$ are sets of new commutative and invertible indeterminates. 
We refer to~\cite[Sec.~4]{finding-ideal-elements} for further details.

\vspace{-0.1cm}
\subsection{Homogenization}
\label{sec:homogenization}
\vspace{-0.1cm}

Recall that given a polynomial $f = \sum_{i=1}^{d} c_{i}w_{i} \in R[X]\langle Y\rangle$, its degree $\deg(f)$ is the maximum length of the monomials $w_i$ appearing in $f$, \ie, $\deg(f) = \max_i |w_i|$.
It is called homogeneous if all its terms have the same length.
The homogenized polynomial $f^{h} \in R[X,h]\langle Y \rangle$ is
$  f^{h} = \sum_{i=1}^{d} c_{i}h^{\deg(f) - |w_{i}|} w_{i}$,
it is a homogeneous polynomial, and evaluating it at $h=1$ yields back $f$.
In that construction, the homogenization variable $h$ commutes with all the variables.

If $R[X,h]\langle Y \rangle$ is ordered by a graded ordering and $h$ is treated as smaller than all other variables, evaluating a (signature) Gröbner basis of $(f_1^{h},\dots,f_{r}^{h})$ at $h=1$ yields a Gröbner basis of $(f_1,\dots,f_{r})$.

It is frequently preferrable to work with homogeneous polynomials when computing Gröbner bases, in both the commutative and noncommutative case.
Specifically in the context of signatures, they open the possibility to use more efficient orderings for the F5 criterion.
More precisely, the fact that one does not know in advance which signature realizes the maximum in the second condition of Corollary~\ref{cor:F5-criterion} makes it expensive to use the criterion.

This can be partially remedied by considering module orders which make this comparison easy, such as the \emph{position-over-term} (PoT) ordering, first comparing the index $\ind(\sigma) \coloneqq i$ of a signature $\sigma = a\varepsilon_{i}b$ before comparing the terms $a$ and $b$.
However, to fully utilize the F5 criterion in this case, elements have to be processed by increasing signatures, which does not constitute a fair selection strategy in the noncommutative setting if PoT is used.
An alternative can be to decouple the selection strategy from the module ordering, but then one cannot expect that all elements necessary to use the F5 criterion will be present in time for its use.

Another possibility, when dealing with homogeneous polynomials and a graded monomial ordering, is the \emph{degree-over-position-over-term} (DoPoT) order.
This order defines the degree of a signature $a\varepsilon_{i}b$ as $|a| + \deg(f_{i}) + |b|$.
In the homogeneous case, given a labeled polynomial $f^{[\alpha]}$, $\deg(f) = \deg(\sig(\alpha))$.
In the inhomogeneous case, it only holds that $\deg(f) \leq \deg(\sig(\alpha))$. 
The DoPoT ordering first compares the degree of the signatures, then the index $i$, and finally the terms $a$ and $b$.
This ordering is fair, and can thus be used as a selection strategy in the algorithm, by always picking the element with smallest signature in the queue $P$ in line~\ref{algoline:sig-select}.
Moreover, this ordering makes it possible to verify the conditions of the F5 criterion easily, and the incremental calculation ensures that all the required signatures are available when applying the criterion.


\begin{corollary}[F5 criterion optimised]
Assume that the generators $f_1,\dots,f_r$ of the labeled module $I\Sig$ are homogeneous and that DoPoT is used as a module ordering with a graded term ordering.

Let $g^{[\beta]} \in G\Sig$.
For all $m \in M(A)$ and $j > \ind(\beta)$, the module terms $\lt(g)m\varepsilon_{j}$ and $\varepsilon_{j}m\lt(g)$ are signatures of trivial syzygies.

If, additionally, elements in Algorithm~\ref{algo:labeledGB} are processed by increasing signature and all ambiguities with signature divisible by a signature as above are discarded, then all ambiguities whose signature is divisible by the signature of a trivial syzygy
$\sigma = g_2 m \beta_1 - \beta_2 m g_1$, with $\ind(\beta_1) \neq \ind(\beta_2)$,
are discarded in this way.
\end{corollary}
\begin{proof}
  They are the signature of the trivial syzygies $gm\varepsilon_{j} - \beta mf_{j}$ and $\varepsilon_{j}mg - f_{j}m\beta$ respectively, observing that both members of a trivial syzygy must have the same degree.

  For the second part, the signature of the trivial syzygy $\sigma$ is $\sig(g_2m\beta_{1})$ if $\ind(\beta_{1})>\ind(\beta_{2})$ and $-\sig(\beta_2 m g_1)$ otherwise.
  Both cases are handled similarly, so assume that we are in the first one.
  Then the signature of $\sigma$ is a discarded signature, obtained from $g_2^{[\beta_{2}]}$.
  It remains to prove $\sig(\beta_{2}) \prec \siga(a)$, to ensure that $g_{2}^{[\beta_{2}]}$ was computed in time for discarding $a$.
  Since $\siga(a)$ is divisible by $\sig(\sigma) \succ \sig(\beta_2 m g_1)$, we have $\deg(\siga(a)) \geq \deg(\sig(\beta_{2}))$. 
  With $\ind(\siga(a)) = \ind(\beta_{1}) > \ind(\beta_{2})$, this yields $\sig(\beta_{2}) \prec \siga(a)$.
\end{proof}

\section{Experimental results}

We have written a prototype implementation\footnote{\label{foot}Available at \url{https://clemenshofstadler.com/software/}} of Algorithm~\ref{algo:labeledGB} in \textsc{SageMath}
for the case when $R$ is a field including the criteria for S-polynomial elimination discussed in Section~\ref{sec:elimination} (G-polynomials are redundant over fields).
We use it to compare the mixed algebra setting to other (more naive) approaches for computing noncommutative (signature) Gr\"obner bases involving some commutative variables.
More precisely, we compare Algorithm~\ref{algo:labeledGB} to the following two approaches:
\begin{enumerate*}
	\item classical Gr\"obner basis computations in the free algebra where commutator relations are added explicitly to the generators of an ideal;
	\item signature Gr\"obner basis computations in the free algebra where commutator relations are added explicitly to the generators of an ideal but are given a trivial signature $0$ so that $\s$-reductions by these relations are always possible. 
\end{enumerate*}

For the classical Gr\"obner basis computations we use \textsc{Singular:Letterplace}~\cite{letterplace} and for the naive signature-based computations we use our \textsc{SageMath} package \texttt{SignatureGB}\textsuperscript{\ref{foot}}.
In Table~\ref{table:comparison}, we report on the number of polynomials reduced, the number of zero reductions and the size of the resulting (signature) Gr\"obner basis for the following benchmark examples.

\begin{itemize}
\item Example \texttt{ufn1h} is from~\cite{letterplace-old} and described there.
	It concerns a homogeneous ideal in $\QQ[h]\< a,b,c,d>$.
		
	\item Iwahori-Hecke algebras~\cite{hump90}, from~\cite[Ex.~31]{LMAZ20}:\\
	 $	\texttt{ih} = (x^2 + hx-qx - hq, y^2 + hy-qy -hq, z^2 + hz-qz - hq,$\\
		$zx-xz, yxy - xyx, zyz-yzy, h^2 - qq^{-1}) \subseteq \QQ[q,q^{-1},h]\< x,y,z>.
	 $
	 
	 \item Homogenization of the relations of the discrete Heisenberg group $\< x,y,z \mid z = xyx^{-1}y^{-1}, xz = zx, yz = zy>$:\\
	 $	\texttt{heis} = (h^3 z - xyx^{-1}y^{-1}, h^2 - zz^{-1}, h^2 - xx^{-1}, h^2 - x^{-1}x,
		h^2 - yy^{-1}, h^2 - y^{-1}y) \subseteq \QQ[z,z^{-1},h]\<x,x^{-1},y,y^{-1}>$.
	 
\end{itemize}

For each of these homogeneous ideals, we compute truncated (signature) Gr\"obner bases up to a fixed degree.
The used degree bounds are indicated by the number after the ``\texttt{--}'' in the names of the examples in Table~\ref{table:comparison}.
For all examples, a degree-lexicographic monomial ordering is used, in combination with DoPoT for the signature-based computations.

\begin{table*}
\centering
\begin{tabular}{c|SSSSSSSSS} 
 \toprule
\multirow{2}{*}{Example} &\multicolumn{3}{c}{classical Gr\"obner basis} & \multicolumn{3}{c}{naive signature basis} & \multicolumn{3}{c}{Algorithm~\ref{algo:labeledGB}} \\
& {reductions} & {red.~to 0} & {size} & {reductions} & {red.~to 0} & {size} & {reductions} & {red.~to 0} & {size}  \\
 \midrule
\texttt{ufn1h-7} & $\num{877}$ &$\num{737}$ &$\num{140}$ &$\num{1784}$ &$\num{401}$ &$\num{1343}$ &$\num{168}$ &$\num{73}$ &$\num{95}$ \\
\texttt{ufn1h-8} &$\num{1447}$ &$\num{1243}$ &$\num{204}$ &$\num{4594}$ &$\num{922}$ &$\num{3556}$ &$\num{270}$ &$\num{120}$ &$\num{150}$ \\
\texttt{ih-7} &$\num{759}$ &$\num{658}$ &$\num{101}$ &$\num{1690}$ &$\num{158}$ &$\num{1363}$ &$\num{31}$ &$\num{7}$ &$\num{24}$ \\
\texttt{ih-8} &$\num{1059}$ &$\num{937}$ &$\num{122}$ &$\num{3635}$ &$\num{289}$ &$\num{2857}$ &$\num{35}$ &$\num{8}$ &$\num{27}$ \\
\texttt{heis-8} &$\num{1167}$ &$\num{823}$ &$\num{344}$ &$\num{18418}$ &$\num{755}$ &$\num{17431}$ &$\num{22}$ &$\num{4}$ &$\num{18}$ \\
\texttt{heis-9} & $\num{4002}$ & $\num{2947}$ &$\num{1055}$ & $\num{63734}$ & $\num{2476}$ & $\num{60510}$ & $\num{48}$ &$\num{14}$ &$\num{34}$ \\
  \bottomrule
\end{tabular}
\caption{}
\label{table:comparison}
\vspace*{-22pt}
\end{table*}

As Table~\ref{table:comparison} shows, Algorithm~\ref{algo:labeledGB} has to perform a lot less reductions and yields a lot smaller outputs.
The main reason for this behaviour is that in the classical approaches the commutator relations become oriented reduction rules, which makes them less flexible.
This causes a lot of computations that are avoided in the mixed algebra.
Additionally, the naive signature-based computation suffers from the fact that the commutator relations are only ``visible'' on the polynomial level but not on the signature level.
Hence, the elimination criteria cannot be exploited fully because they miss this crucial information.
In contrast to this, in the mixed algebra setting, the information about the commutative variables is also directly propagated to the signatures.


\begin{acks}
We thank the reviewers for their careful reading and their remarks.
\end{acks}

\bibliographystyle{ACM-Reference-Format}

\begin{thebibliography}{47}


\ifx \showCODEN    \undefined \def \showCODEN     #1{\unskip}     \fi
\ifx \showDOI      \undefined \def \showDOI       #1{#1}\fi
\ifx \showISBNx    \undefined \def \showISBNx     #1{\unskip}     \fi
\ifx \showISBNxiii \undefined \def \showISBNxiii  #1{\unskip}     \fi
\ifx \showISSN     \undefined \def \showISSN      #1{\unskip}     \fi
\ifx \showLCCN     \undefined \def \showLCCN      #1{\unskip}     \fi
\ifx \shownote     \undefined \def \shownote      #1{#1}          \fi
\ifx \showarticletitle \undefined \def \showarticletitle #1{#1}   \fi
\ifx \showURL      \undefined \def \showURL       {\relax}        \fi
\providecommand\bibfield[2]{#2}
\providecommand\bibinfo[2]{#2}
\providecommand\natexlab[1]{#1}
\providecommand\showeprint[2][]{arXiv:#2}

\bibitem[Bergman(1978)]%
        {ber78}
\bibfield{author}{\bibinfo{person}{George~M. Bergman}.}
  \bibinfo{year}{1978}\natexlab{}.
\newblock \showarticletitle{The diamond lemma for ring theory}.
\newblock \bibinfo{journal}{\emph{Adv. in Math.}} \bibinfo{volume}{29},
  \bibinfo{number}{2} (\bibinfo{year}{1978}), \bibinfo{pages}{178--218}.
\newblock
\showISSN{0001-8708}


\bibitem[Borges and Borges(1998)]%
        {Borges_1998}
\bibfield{author}{\bibinfo{person}{Miguel~Angel Borges} {and}
  \bibinfo{person}{Mijail Borges}.} \bibinfo{year}{1998}\natexlab{}.
\newblock \showarticletitle{Gröbner Bases Property on Elimination Ideal in the
  Noncommutative Case}.
\newblock In \bibinfo{booktitle}{\emph{Gröbner Bases and Applications}}.
  \bibinfo{publisher}{Cambridge University Press}, \bibinfo{pages}{323--337}.
\newblock


\bibitem[Buchberger(1965)]%
        {buchberger-1965}
\bibfield{author}{\bibinfo{person}{Bruno Buchberger}.}
  \bibinfo{year}{1965}\natexlab{}.
\newblock \emph{\bibinfo{title}{{E}in {A}lgorithmus zum {A}uffinden der
  {B}asiselemente des {R}estklassenringes nach einem nulldimensionalen
  {P}olynomideal}}.
\newblock \bibinfo{thesistype}{Ph.\,D. Dissertation}.
  \bibinfo{school}{University of Innsbruck, Austria}.
\newblock


\bibitem[Ceria and Mora(2017a)]%
        {Spic}
\bibfield{author}{\bibinfo{person}{Michela Ceria} {and} \bibinfo{person}{Teo
  Mora}.} \bibinfo{year}{2017}\natexlab{a}.
\newblock \showarticletitle{{Buchberger-Weispfenning theory for effective
  associative rings}}.
\newblock \bibinfo{journal}{\emph{J. Symb. Comput.}}  \bibinfo{volume}{83}
  (\bibinfo{year}{2017}), \bibinfo{pages}{112--146}.
\newblock


\bibitem[Ceria and Mora(2017b)]%
        {labOrE}
\bibfield{author}{\bibinfo{person}{Michela Ceria} {and} \bibinfo{person}{Teo
  Mora}.} \bibinfo{year}{2017}\natexlab{b}.
\newblock \showarticletitle{{Buchberger{\textendash}Zacharias Theory of
  multivariate Ore extensions}}.
\newblock \bibinfo{journal}{\emph{J. Pure Appl. Algebra}}
  \bibinfo{volume}{221}, \bibinfo{number}{12} (\bibinfo{year}{2017}),
  \bibinfo{pages}{2974--3026}.
\newblock


\bibitem[Chenavier et~al\mbox{.}(2020)]%
        {CHRR20}
\bibfield{author}{\bibinfo{person}{Cyrille Chenavier}, \bibinfo{person}{Clemens
  Hofstadler}, \bibinfo{person}{Clemens~G. Raab}, {and} \bibinfo{person}{Georg
  Regensburger}.} \bibinfo{year}{2020}\natexlab{}.
\newblock \showarticletitle{Compatible rewriting of noncommutative polynomials
  for proving operator identities}. In \bibinfo{booktitle}{\emph{Proceedings of
  ISSAC 2020}}. \bibinfo{pages}{83--90}.
\newblock


\bibitem[Cohn(1985)]%
        {cohn-free-rings-and-their-relations}
\bibfield{author}{\bibinfo{person}{Paul~M. Cohn}.}
  \bibinfo{year}{1985}\natexlab{}.
\newblock \bibinfo{booktitle}{\emph{{Free rings and their relations}}
  (\bibinfo{edition}{2nd} ed.)}.
\newblock \bibinfo{publisher}{Academic Press}.
\newblock


\bibitem[Eder and Faug{\`e}re(2017)]%
        {eder:2017:survey}
\bibfield{author}{\bibinfo{person}{Christian Eder} {and}
  \bibinfo{person}{Jean-Charles Faug{\`e}re}.} \bibinfo{year}{2017}\natexlab{}.
\newblock \showarticletitle{A {S}urvey on {S}ignature-based {A}lgorithms for
  {C}omputing {G}r{\"o}bner {B}ases}.
\newblock \bibinfo{journal}{\emph{J. Symbolic Comput.}}  \bibinfo{volume}{80}
  (\bibinfo{year}{2017}), \bibinfo{pages}{719--784}.
\newblock


\bibitem[Eder et~al\mbox{.}(2023)]%
        {eder2022signature}
\bibfield{author}{\bibinfo{person}{Christian Eder}, \bibinfo{person}{Pierre
  Lairez}, \bibinfo{person}{Rafael Mohr}, {and} \bibinfo{person}{Mohab Safey
  El~Din}.} \bibinfo{year}{2023}\natexlab{}.
\newblock \showarticletitle{A signature-based algorithm for computing the
  nondegenerate locus of a polynomial system}.
\newblock \bibinfo{journal}{\emph{J. Symbolic Comput.}}  \bibinfo{volume}{119}
  (\bibinfo{year}{2023}), \bibinfo{pages}{1--21}.
\newblock


\bibitem[Eder et~al\mbox{.}(2017)]%
        {eder-pfister-popescu}
\bibfield{author}{\bibinfo{person}{Christian Eder}, \bibinfo{person}{Gerhard
  Pfister}, {and} \bibinfo{person}{Adrian Popescu}.}
  \bibinfo{year}{2017}\natexlab{}.
\newblock \showarticletitle{On signature-based {G}r\"{o}bner bases over
  {E}uclidean rings}. In \bibinfo{booktitle}{\emph{{P}roceedings of ISSAC
  2017}}. \bibinfo{pages}{141--148}.
\newblock


\bibitem[Faug{\`e}re(2002)]%
        {Faugere-2002-F5}
\bibfield{author}{\bibinfo{person}{Jean-Charles Faug{\`e}re}.}
  \bibinfo{year}{2002}\natexlab{}.
\newblock \showarticletitle{A new efficient algorithm for computing {G}r\"obner
  bases without reduction to zero {$(F_5)$}}. In
  \bibinfo{booktitle}{\emph{Proceedings of ISSAC 2002}}.
  \bibinfo{pages}{75--83}.
\newblock


\bibitem[Francis and Verron(2021)]%
        {francis-verron}
\bibfield{author}{\bibinfo{person}{Maria Francis} {and}
  \bibinfo{person}{Thibaut Verron}.} \bibinfo{year}{2021}\natexlab{}.
\newblock \showarticletitle{On two signature variants of {B}uchberger's
  algorithm over principal ideal domains}. In
  \bibinfo{booktitle}{\emph{Proceedings of ISSAC 2021}}.
  \bibinfo{pages}{139--146}.
\newblock


\bibitem[Galligo(1985)]%
        {Galligo-1985-some-algorithmic-questions}
\bibfield{author}{\bibinfo{person}{André Galligo}.}
  \bibinfo{year}{1985}\natexlab{}.
\newblock \showarticletitle{Some algorithmic questions on ideals of
  differential operators}.
\newblock \bibinfo{journal}{\emph{Lecture Notes in Comput. Sci.}}
  (\bibinfo{year}{1985}), \bibinfo{pages}{413–421}.
\newblock
\showISSN{1611-3349}


\bibitem[Gao et~al\mbox{.}(2015)]%
        {GVW}
\bibfield{author}{\bibinfo{person}{Shuhong Gao}, \bibinfo{person}{Frank {Volny
  IV}}, {and} \bibinfo{person}{Mingsheng Wang}.}
  \bibinfo{year}{2015}\natexlab{}.
\newblock \showarticletitle{{A new framework for computing Gröbner bases}}.
\newblock \bibinfo{journal}{\emph{Math. Comput.}} \bibinfo{volume}{85},
  \bibinfo{number}{297} (\bibinfo{year}{2015}), \bibinfo{pages}{449–465}.
\newblock
\showISSN{1088-6842}


\bibitem[Gebauer and M{\"{o}}ller(1986)]%
        {GM2}
\bibfield{author}{\bibinfo{person}{R{\"{u}}diger Gebauer} {and}
  \bibinfo{person}{H.~Michael M{\"{o}}ller}.} \bibinfo{year}{1986}\natexlab{}.
\newblock \showarticletitle{Buchberger's algorithm and staggered linear bases}.
  In \bibinfo{booktitle}{\emph{Proceedings of the Symposium on Symbolic and
  Algebraic Manipulation, {SYMSAC} 1986}}. \bibinfo{pages}{218--221}.
\newblock


\bibitem[Hashemi and Javanbakht(2020)]%
        {HJ}
\bibfield{author}{\bibinfo{person}{Amir Hashemi} {and}
  \bibinfo{person}{Masoumeh Javanbakht}.} \bibinfo{year}{2020}\natexlab{}.
\newblock \showarticletitle{On the construction of staggered linear bases}.
\newblock \bibinfo{journal}{\emph{J. Algebra Appl.}} \bibinfo{volume}{20},
  \bibinfo{number}{08} (\bibinfo{year}{2020}), \bibinfo{pages}{2150132}.
\newblock


\bibitem[Hashemi and M{\"{o}}ller(2023)]%
        {Hashemi}
\bibfield{author}{\bibinfo{person}{Amir Hashemi} {and}
  \bibinfo{person}{H.~Michael M{\"{o}}ller}.} \bibinfo{year}{2023}\natexlab{}.
\newblock \showarticletitle{A new algorithm for computing staggered linear
  bases}.
\newblock \bibinfo{journal}{\emph{J. Symb. Comput.}}  \bibinfo{volume}{117}
  (\bibinfo{year}{2023}), \bibinfo{pages}{1--14}.
\newblock


\bibitem[Helton and Wavrik(1994)]%
        {HW94}
\bibfield{author}{\bibinfo{person}{J.~William Helton} {and}
  \bibinfo{person}{John~J. Wavrik}.} \bibinfo{year}{1994}\natexlab{}.
\newblock \showarticletitle{Rules for computer simplification of the formulas
  in operator model theory and linear systems}.
\newblock In \bibinfo{booktitle}{\emph{Nonselfadjoint operators and related
  topics}}. \bibinfo{publisher}{Springer}, \bibinfo{pages}{325--354}.
\newblock


\bibitem[Hofstadler et~al\mbox{.}(2019)]%
        {HRR19}
\bibfield{author}{\bibinfo{person}{Clemens Hofstadler},
  \bibinfo{person}{Clemens~G. Raab}, {and} \bibinfo{person}{Georg
  Regensburger}.} \bibinfo{year}{2019}\natexlab{}.
\newblock \showarticletitle{{Certifying operator identities via noncommutative
  Gr{\"o}bner bases}}.
\newblock \bibinfo{journal}{\emph{ACM Commun. Comput. Algebra}}
  \bibinfo{volume}{53}, \bibinfo{number}{2} (\bibinfo{year}{2019}),
  \bibinfo{pages}{49--52}.
\newblock


\bibitem[Hofstadler et~al\mbox{.}(2022)]%
        {finding-ideal-elements}
\bibfield{author}{\bibinfo{person}{Clemens Hofstadler},
  \bibinfo{person}{Clemens~G. Raab}, {and} \bibinfo{person}{Georg
  Regensburger}.} \bibinfo{year}{2022}\natexlab{}.
\newblock \showarticletitle{Computing elements of certain form in ideals to
  prove properties of operators}.
\newblock \bibinfo{journal}{\emph{Math. Comput. Sci.}} \bibinfo{volume}{16},
  \bibinfo{number}{2-3} (\bibinfo{year}{2022}), \bibinfo{pages}{17}.
\newblock
\showISSN{1661-8270}


\bibitem[Hofstadler and Verron(2022)]%
        {hofstadler-verron}
\bibfield{author}{\bibinfo{person}{Clemens Hofstadler} {and}
  \bibinfo{person}{Thibaut Verron}.} \bibinfo{year}{2022}\natexlab{}.
\newblock \showarticletitle{{Signature Gröbner bases, bases of syzygies and
  cofactor reconstruction in the free algebra}}.
\newblock \bibinfo{journal}{\emph{J. Symbolic Comput.}}  \bibinfo{volume}{113}
  (\bibinfo{year}{2022}), \bibinfo{pages}{211–241}.
\newblock
\showISSN{0747-7171}


\bibitem[Hofstadler and Verron(2023)]%
        {hofstadler-verron-shortrep}
\bibfield{author}{\bibinfo{person}{Clemens Hofstadler} {and}
  \bibinfo{person}{Thibaut Verron}.} \bibinfo{year}{2023}\natexlab{}.
\newblock \showarticletitle{Short proofs of ideal membership}.
\newblock \bibinfo{journal}{\emph{arXiv preprints}}
  \bibinfo{volume}{arXiv:2302.02832} (\bibinfo{year}{2023}).
\newblock


\bibitem[Humphreys(1990)]%
        {hump90}
\bibfield{author}{\bibinfo{person}{James~E. Humphreys}.}
  \bibinfo{year}{1990}\natexlab{}.
\newblock \bibinfo{booktitle}{\emph{Reflection groups and {C}oxeter groups}}.
  \bibinfo{series}{Cambridge Stud. Adv. Math.}, Vol.~\bibinfo{volume}{29}.
\newblock \bibinfo{publisher}{Cambridge University Press, Cambridge}. xii+204
  pages.
\newblock


\bibitem[Kandri-Rody and Kapur(1988)]%
        {Kandri-Rody-Kapur}
\bibfield{author}{\bibinfo{person}{{\"A}bdelilah Kandri-Rody} {and}
  \bibinfo{person}{Deepak Kapur}.} \bibinfo{year}{1988}\natexlab{}.
\newblock \showarticletitle{{Computing a {G}r\"obner Basis of a Polynomial
  Ideal over a {E}uclidean Domain}}.
\newblock \bibinfo{journal}{\emph{J. Symbolic Comput.}} \bibinfo{volume}{6},
  \bibinfo{number}{1} (\bibinfo{year}{1988}), \bibinfo{pages}{37--57}.
\newblock


\bibitem[Kapur and Madlener(1989)]%
        {Kapur-1989-CompletionProcedureComputing}
\bibfield{author}{\bibinfo{person}{Deepak Kapur} {and} \bibinfo{person}{Klaus
  Madlener}.} \bibinfo{year}{1989}\natexlab{}.
\newblock \showarticletitle{A {{Completion Procedure}} for {{Computing}} a
  {{Canonical Basis}} for a $k$-{{Subalgebra}}}.
\newblock In \bibinfo{booktitle}{\emph{Computers and {{Mathematics}}}}.
  \bibinfo{publisher}{{Springer US}}, \bibinfo{pages}{1--11}.
\newblock


\bibitem[La~Scala and Levandovskyy(2009)]%
        {letterplace-old}
\bibfield{author}{\bibinfo{person}{Roberto La~Scala} {and}
  \bibinfo{person}{Viktor Levandovskyy}.} \bibinfo{year}{2009}\natexlab{}.
\newblock \showarticletitle{Letterplace ideals and non-commutative
  {G}r\"{o}bner bases}.
\newblock \bibinfo{journal}{\emph{J. Symbolic Comput.}} \bibinfo{volume}{44},
  \bibinfo{number}{10} (\bibinfo{year}{2009}), \bibinfo{pages}{1374--1393}.
\newblock
\showISSN{0747-7171}


\bibitem[Lairez(2022)]%
        {Lairez-2022-AxiomsForTheory}
\bibfield{author}{\bibinfo{person}{Pierre Lairez}.}
  \bibinfo{year}{2022}\natexlab{}.
\newblock \showarticletitle{Axioms for a theory of signature bases}.
\newblock \bibinfo{journal}{\emph{arXiv preprints}}
  \bibinfo{volume}{arXiv:2210.13788} (\bibinfo{year}{2022}).
\newblock


\bibitem[Levandovskyy et~al\mbox{.}(2023)]%
        {LMAZ20}
\bibfield{author}{\bibinfo{person}{Viktor Levandovskyy},
  \bibinfo{person}{Tobias Metzlaff}, {and} \bibinfo{person}{Karim Abou~Zeid}.}
  \bibinfo{year}{2023}\natexlab{}.
\newblock \showarticletitle{Computing free non-commutative {G}r\"{o}bner bases
  over {$\mathbb{Z}$} with \textsc{Singular:Letterplace}}.
\newblock \bibinfo{journal}{\emph{J. Symbolic Comput.}}  \bibinfo{volume}{115}
  (\bibinfo{year}{2023}), \bibinfo{pages}{201--222}.
\newblock
\showISSN{0747-7171}


\bibitem[Levandovskyy et~al\mbox{.}(2020)]%
        {letterplace}
\bibfield{author}{\bibinfo{person}{Viktor Levandovskyy}, \bibinfo{person}{Hans
  Sch\"{o}nemann}, {and} \bibinfo{person}{Karim Abou~Zeid}.}
  \bibinfo{year}{2020}\natexlab{}.
\newblock \showarticletitle{{\textsc{Letterplace}---a Subsystem of
  \textsc{Singular} for Computations with Free Algebras via Letterplace
  Embedding}}. In \bibinfo{booktitle}{\emph{{P}roceedings of ISSAC 2020}}.
  \bibinfo{pages}{305--311}.
\newblock


\bibitem[Lichtblau(2012)]%
        {Lichtblau}
\bibfield{author}{\bibinfo{person}{Daniel Lichtblau}.}
  \bibinfo{year}{2012}\natexlab{}.
\newblock \showarticletitle{{Effective Computation of Strong {G}r\"obner Bases
  over {E}uclidean Domains}}.
\newblock \bibinfo{journal}{\emph{Illinois J. Math.}} \bibinfo{volume}{56},
  \bibinfo{number}{1} (\bibinfo{year}{2012}), \bibinfo{pages}{177--194}.
\newblock
\showISSN{0019-2082}


\bibitem[Lind and Schmidt(2015)]%
        {lind2015survey}
\bibfield{author}{\bibinfo{person}{Douglas Lind} {and} \bibinfo{person}{Klaus
  Schmidt}.} \bibinfo{year}{2015}\natexlab{}.
\newblock \showarticletitle{{A survey of algebraic actions of the discrete
  Heisenberg group}}.
\newblock \bibinfo{journal}{\emph{Russ. Math. Surveys}} \bibinfo{volume}{70},
  \bibinfo{number}{4} (\bibinfo{year}{2015}), \bibinfo{pages}{657}.
\newblock


\bibitem[Mikhalev and Zolotykh(1998)]%
        {mikhalev}
\bibfield{author}{\bibinfo{person}{Alexander~A. Mikhalev} {and}
  \bibinfo{person}{Andrej~A. Zolotykh}.} \bibinfo{year}{1998}\natexlab{}.
\newblock \showarticletitle{Standard {G}r\"{o}bner-{S}hirshov bases of free
  algebras over rings. {I}. {F}ree associative algebras}.
\newblock \bibinfo{journal}{\emph{Internat. J. Algebra Comput.}}
  \bibinfo{volume}{8}, \bibinfo{number}{6} (\bibinfo{year}{1998}),
  \bibinfo{pages}{689--726}.
\newblock
\showISSN{0218-1967}


\bibitem[{M}{\"o}ller(1988)]%
        {Moller:1988:grobnerrings2}
\bibfield{author}{\bibinfo{person}{{H}.~Michael {M}{\"o}ller}.}
  \bibinfo{year}{1988}\natexlab{}.
\newblock \showarticletitle{{O}n the {C}onstruction of {G}r{\"o}bner {B}ases
  using {S}yzygies}.
\newblock \bibinfo{journal}{\emph{J. Symbolic Comput}} \bibinfo{volume}{6},
  \bibinfo{number}{2-3} (\bibinfo{year}{1988}), \bibinfo{pages}{345--359}.
\newblock


\bibitem[{M}{\"o}ller and {M}ora(1986)]%
        {Moller:1986:generalizationtomodules}
\bibfield{author}{\bibinfo{person}{H.~Michale {M}{\"o}ller} {and}
  \bibinfo{person}{Ferdinando {M}ora}.} \bibinfo{year}{1986}\natexlab{}.
\newblock \showarticletitle{New {C}onstructive {M}ethods in {C}lassical {I}deal
  {T}heory}.
\newblock \bibinfo{journal}{\emph{J. Algebra}} \bibinfo{volume}{100},
  \bibinfo{number}{1} (\bibinfo{year}{1986}), \bibinfo{pages}{138 --178}.
\newblock


\bibitem[M{\"{o}}ller et~al\mbox{.}(1992)]%
        {MMT}
\bibfield{author}{\bibinfo{person}{H.~Michael M{\"{o}}ller},
  \bibinfo{person}{Teo Mora}, {and} \bibinfo{person}{Carlo Traverso}.}
  \bibinfo{year}{1992}\natexlab{}.
\newblock \showarticletitle{Gr{\"{o}}bner Bases Computation Using Syzygies}. In
  \bibinfo{booktitle}{\emph{Proceedings of ISSAC 1992}}.
  \bibinfo{pages}{320--328}.
\newblock
\showISBNx{0-89791-489-9}


\bibitem[Mora(1985)]%
        {Mor85}
\bibfield{author}{\bibinfo{person}{Ferdinando Mora}.}
  \bibinfo{year}{1985}\natexlab{}.
\newblock \showarticletitle{Gr{\"o}bner bases for non-commutative polynomial
  rings}. In \bibinfo{booktitle}{\emph{International Conference on Applied
  Algebra, Algebraic Algorithms, and Error-Correcting Codes}}. Springer,
  \bibinfo{pages}{353--362}.
\newblock


\bibitem[Mora(2015)]%
        {NG4}
\bibfield{author}{\bibinfo{person}{Ferdinando Mora}.}
  \bibinfo{year}{2015}\natexlab{}.
\newblock \showarticletitle{{De Nugis Groebnerialium 4: Zacharias, Spears,
  M{\"{o}}ller}}. In \bibinfo{booktitle}{\emph{Proceedings of {ISSAC} 2015}}.
  \bibinfo{pages}{283--290}.
\newblock


\bibitem[Mora(1994)]%
        {Mora-1994-IntroductionToCommutative}
\bibfield{author}{\bibinfo{person}{Teo Mora}.} \bibinfo{year}{1994}\natexlab{}.
\newblock \showarticletitle{{An Introduction to Commutative and Noncommutative
  Gr{\"{o}}bner Bases}}.
\newblock \bibinfo{journal}{\emph{Theor. Comput. Sci.}} \bibinfo{volume}{134},
  \bibinfo{number}{1} (\bibinfo{year}{1994}), \bibinfo{pages}{131--173}.
\newblock


\bibitem[Mora(2016)]%
        {SPES4}
\bibfield{author}{\bibinfo{person}{Teo Mora}.} \bibinfo{year}{2016}\natexlab{}.
\newblock \bibinfo{booktitle}{\emph{Solving Polynomial Equation Systems {IV}}}.
\newblock \bibinfo{publisher}{Cambridge University Press}.
\newblock


\bibitem[Nguefack and Pola(2020)]%
        {Cameroun}
\bibfield{author}{\bibinfo{person}{Bertrand Nguefack} {and}
  \bibinfo{person}{Emmanuel Pola}.} \bibinfo{year}{2020}\natexlab{}.
\newblock \showarticletitle{{Effective Buchberger-Zacharias-Weispfenning theory
  of skew polynomial extensions of subbilateral coherent rings}}.
\newblock \bibinfo{journal}{\emph{J. Symb. Comput.}}  \bibinfo{volume}{99}
  (\bibinfo{year}{2020}), \bibinfo{pages}{50--107}.
\newblock


\bibitem[Nordbeck(1998)]%
        {Nordbeck_1998}
\bibfield{author}{\bibinfo{person}{Patrik Nordbeck}.}
  \bibinfo{year}{1998}\natexlab{}.
\newblock \showarticletitle{{On some Basic Applications of Gröbner Bases in
  Non-commutative Polynomial Rings}}.
\newblock In \bibinfo{booktitle}{\emph{Gröbner Bases and Applications}}.
  \bibinfo{publisher}{Cambridge University Press}, \bibinfo{pages}{463--472}.
\newblock


\bibitem[Ollivier(1991)]%
        {Ollivier}
\bibfield{author}{\bibinfo{person}{Fran{\c{c}}ois Ollivier}.}
  \bibinfo{year}{1991}\natexlab{}.
\newblock \showarticletitle{Canonical Bases: Relations with Standard Bases,
  Finiteness Conditions and Application to Tame Automorphisms}.
\newblock In \bibinfo{booktitle}{\emph{Effective Methods in Algebraic
  Geometry}}. \bibinfo{publisher}{Birkhäuser Boston},
  \bibinfo{pages}{379--400}.
\newblock


\bibitem[Pan(1989)]%
        {Pan:Dbases}
\bibfield{author}{\bibinfo{person}{Luquan Pan}.}
  \bibinfo{year}{1989}\natexlab{}.
\newblock \showarticletitle{{On the {D}-bases of Polynomial Ideals over
  Principal Ideal Domains}}.
\newblock \bibinfo{journal}{\emph{J. Symbolic Comput.}} \bibinfo{volume}{7},
  \bibinfo{number}{1} (\bibinfo{year}{1989}), \bibinfo{pages}{55--69}.
\newblock
\showISSN{0747-7171}


\bibitem[Pritchard(1996)]%
        {Pritchard}
\bibfield{author}{\bibinfo{person}{F.~Leon Pritchard}.}
  \bibinfo{year}{1996}\natexlab{}.
\newblock \showarticletitle{{The Ideal Membership Problem in Non-Commutative
  Polynomial Rings}}.
\newblock \bibinfo{journal}{\emph{J. Symb. Comput.}} \bibinfo{volume}{22},
  \bibinfo{number}{1} (\bibinfo{year}{1996}), \bibinfo{pages}{27--48}.
\newblock


\bibitem[Raab et~al\mbox{.}(2021)]%
        {raab2021formal}
\bibfield{author}{\bibinfo{person}{Clemens~G. Raab}, \bibinfo{person}{Georg
  Regensburger}, {and} \bibinfo{person}{Jamal Hossein~Poor}.}
  \bibinfo{year}{2021}\natexlab{}.
\newblock \showarticletitle{Formal proofs of operator identities by a single
  formal computation}.
\newblock \bibinfo{journal}{\emph{J. Pure Appl. Algebra}}
  \bibinfo{volume}{225}, \bibinfo{number}{5} (\bibinfo{year}{2021}),
  \bibinfo{pages}{106564}.
\newblock


\bibitem[Schmitz and Levandovskyy(2020)]%
        {SL20}
\bibfield{author}{\bibinfo{person}{Leonard Schmitz} {and}
  \bibinfo{person}{Viktor Levandovskyy}.} \bibinfo{year}{2020}\natexlab{}.
\newblock \showarticletitle{{Formally Verifying Proofs for Algebraic Identities
  of Matrices}}. In \bibinfo{booktitle}{\emph{International Conference on
  Intelligent Computer Mathematics}}. Springer, \bibinfo{pages}{222--236}.
\newblock


\bibitem[Sun et~al\mbox{.}(2012)]%
        {Sun-2012-signature-groebner-solvable}
\bibfield{author}{\bibinfo{person}{Yao Sun}, \bibinfo{person}{Dingkang Wang},
  \bibinfo{person}{Xiaodong Ma}, {and} \bibinfo{person}{Yang Zhang}.}
  \bibinfo{year}{2012}\natexlab{}.
\newblock \showarticletitle{{A Signature-Based Algorithm for Computing Gröbner
  Bases in Solvable Polynomial Algebras}}. In
  \bibinfo{booktitle}{\emph{Proceedings of ISSAC 2012}}.
  \bibinfo{pages}{351--358}.
\newblock


\end{thebibliography}



\end{document}